\documentclass[11pt]{amsart}
\usepackage{amsmath,amssymb, xypic,verbatim,eucal, amscd,color,mathrsfs}
\usepackage{mathtools}
\usepackage{graphicx}
\usepackage{colordvi}
\usepackage{mathdots}
\usepackage[colorlinks=true, pdfstartview=FitV,
linkcolor=cyan, citecolor=red, urlcolor=blue]{hyperref}
\usepackage{amsfonts,latexsym,bm}
\usepackage{tikz-cd}

\usepackage{yhmath}
\usepackage{accents}

\usepackage[all,arc]{xy}
\input xy

\numberwithin{equation}{section}

\theoremstyle{plain}
				
\newtheorem{theorem}{Theorem}[section]				
\newtheorem{proposition}[theorem]{Proposition}		
\newtheorem{corollary}[theorem]{Corollary}
\newtheorem{lemma}[theorem]{Lemma}

\theoremstyle{definition}

\newtheorem{definition}[theorem]{Definition}
\newtheorem{remark}[theorem]{Remark}
\newtheorem{question}[theorem]{Question}

\newcommand{\CBbb}{\mathbb C}

\newcommand{\Pcal}{\mathcal P}

\newcommand{\Tcal}{\mathcal T}

\newcommand{\gfrak}{\mathfrak g}
\newcommand{\hfrak}{\mathfrak h}

\newcommand{\Lscr}{\mathscr L}

\newcommand{\SL}{\mathsf{SL}}

\newcommand{\GL}{\mathsf{GL}}

\DeclareMathOperator{\id}{id}

\DeclareMathOperator{\Sym}{Sym}

\DeclareMathOperator{\ad}{ad}

\newcommand{\lra}{\longrightarrow}

\newcommand\norder[1]{\vcentcolon\mathrel{#1}\vcentcolon}

\newcommand{\Gr}{{\rm Gr}}

\newcommand{\isorightarrow}{\xrightarrow{
   \,\smash{\raisebox{-0.5ex}{\ensuremath{\sim}}}\,}}

\AtBeginDocument{%
	\def\MR#1{}
}

\begin{document}

\title[Geometrization of the TUY/WZW/KZ connection]{Geometrization of the TUY/WZW/KZ 
connection}

	\author[Biswas]{Indranil Biswas}
	\address{Mathematics Department, Shiv Nadar University, NH91, Tehsil Dadri,
		Greater Noida, Uttar Pradesh 201314, India}
	\email{indranil.biswas@snu.edu.in, indranil29@gmail.com}
	\author[Mukhopadhyay]{Swarnava Mukhopadhyay}
	\address{School of Mathematics,
		Tata Institute of Fundamental Research,
		Mumbai-400005, India}	\email{swarnava@math.tifr.res.in}
	\author[Wentworth]{Richard Wentworth}
	\address{Department of Mathematics,
		University of Maryland,
		College Park, MD 20742, USA}
	\email{raw@umd.edu}
	
	\thanks{I.B. 
is	supported in part by a J.C. Bose fellowship and both I.B. and S.M. are also partly supported by DAE, India
		under project no. 1303/3/2019/R\&D/IIDAE/13820. S.M.
		received additional funding from the Science and Engineering Research
		Board, India (SRG/2019/000513). R.W. is 
		supported in part by National Science Foundation grants DMS-1906403
        and DMS-2204346.}

	
\begin{abstract}
Given a simple, simply connected, complex algebraic group $G$, a flat projective connection 
on the bundle of nonabelian theta functions on the moduli space of semistable parabolic 
$G$-bundles over any family of smooth projective curves with marked points was constructed 
by the authors in an earlier paper. Here, it is shown
that the identification between the bundle of nonabelian theta 
functions and the bundle of WZW conformal blocks
is flat with respect to this connection and the one constructed by Tsuchiya-Ueno-Yamada. 
As an application, we give a 
geometric construction of the Knizhnik-Zamolodchikov connection on the trivial bundle over 
the configuration space of points in the projective line whose typical fiber is the 
space of invariants of tensor product of representations.
\end{abstract}

\maketitle

\tableofcontents

\allowdisplaybreaks

\section{Introduction}

The Wess-Zumino-Witten (WZW) model \cite{Novikov:82, Witten:84} is a 
cornerstone of two dimensional rational conformal field theories
\cite{BPZ:84, MooreSeiberg:90}.
The WZW 
conformal blocks were constructed mathematically by Tsuchiya-Ueno-Yamada \cite{TUY:89}.
Let $\widehat{\frak{g}}$ be 
an affine Lie algebra and
$(C,\,\bm p)$ a smooth curve $C$ with $n$-distinct marked points $\bm 
p\,=\,(p_1,\,\cdots,\, p_n)$. Choose formal coordinates $\bm
\xi\,=\,(\xi_1,\,\cdots,\, \xi_n)$ around $\bm p$, and using these coordinates assign a copy of 
$\widehat{\frak{g}}$ to each point $p_i$. Fix a positive integer $\ell$.
Then for any 
choice of $n$-tuple of integrable highest weights
$\vec{\lambda}\,=\,(\lambda_1,\,\cdots, \,\lambda_n)$ of level $\ell$, 
the construction in \cite{TUY:89} associates a finite dimensional vector space 
$\mathcal{V}_{\vec{\lambda}}^{\dagger}(C, \bm p,\bm \xi, \frak{g},\ell)$ to the data 
$(C,\,\bm p,\, \bm \xi)$. For 
a family of smooth curves $\pi\,:\,\mathcal{C}\,\to\, S$ with 
$n$-distinct sections $\bm p$,
these vector spaces patch together to produce a coherent sheaf 
$\mathcal{V}^{\dagger}_{\vec{\lambda}}(\frak{g}, \ell)\,\to\, S$.
The Sugawara construction \cite{TUY:89} endows this sheaf 
with the structure of 
a twisted $\mathcal{D}$-module, and hence 
$\mathcal{V}_{\vec{\lambda}}^{\dagger}(\frak{g}, \ell)$ is actually a holomorphic vector bundle. 
The authors of \cite{TUY:89} show that this vector bundle extends to the 
Deligne-Mumford-Knudsen compactification $\overline{\mathcal{M}}_{g,n}'$ of the moduli 
spaces of $n$-pointed curves $\mathcal{M}_{g,n}'$
with chosen formal coordinates. Moreover, the flat projective 
connection on the interior $\mathcal{M}_{g,n}'$ extends to a flat projective connection 
with logarithmic singularities over $\overline{\mathcal{M}}_{g,n}'$. 
 The bundle $\mathcal{V}^{\dagger}_{\vec{\lambda}}(\frak{g},
\ell)\,\to\, S$ of conformal blocks is sometimes called  the \emph{Friedan-Shenker
bundle}. We refer to the above mentioned
flat projective connection on $\mathcal{V}^{\dagger}_{\vec{\lambda}}(\frak{g},
\ell)\,\to\, S$ as the {\em WZW/TUY connection}.
	
Later, Tsuchimoto \cite{Tsuchimoto} gave a coordinate free construction of the bundle of 
conformal blocks and showed that  it
descends  to a vector bundle 
$\mathbb{V}_{\vec{\lambda}}^{\dagger}(\frak{g},\ell)$ on the Deligne-Mumford-Knudsen moduli 
space $\overline{\mathcal{M}}_{g,n}$ of $n$-pointed stable nodal curves (cf.\ Fakhruddin \cite{Fakhruddin:12}). The flat 
projective connection also descends to
a projective connection with logarithmic singularities.
In other words, there is a projectively flat isomorphism between
the conformal blocks $\mathcal{V}_{\vec{\lambda}}^{\dagger}(\frak{g},\ell)$
and the pullback $F^*\mathbb{V}_{\vec{\lambda}}^{\dagger}(\frak{g},\ell)$
under the 
natural forgetful map $F\,:\, \overline{\mathcal{M}}_{g,n}'\,\to\, 
\overline{\mathcal{M}}_{g,n}$. We refer the reader to Section 
\ref{sec:conformalblocksbasic} for a construction of conformal blocks and
to Section 
\ref{sec:connection} for the construction of the WZW/TUY connection.

We now discuss how conformal blocks are related to moduli spaces of bundles on 
curves.
The moduli space $M_G(C)$ of principal bundles, with a reductive structure group $G$, on a 
smooth projective curve $C$, provides a natural nonabelian generalization of the Jacobian 
variety $J(C)$, which parametrizes line bundles of degree zero on $C$. The moduli space of 
(semistable) principal $G$-bundles on a smooth projective algebraic curve is itself a projective 
variety. It was originally constructed through Geometric Invariant Theory. Its smooth 
locus  parametrizes isomorphism classes of  stable bundles with minimal automorphism groups (see
\cite{BiswasHoffmann:12}), also known as the regularly stable loci. There are various important variations on this
construction.
One can choose marked points $\bm p\,=\,(p_1,\,\cdots,\, p_n)$ on the algebraic curve
$C$ and decorate a principal $G$-bundle $\Pcal$ with a generalized flag structure over $\bm p$, 
leading to the notion of quasi-parabolic bundles. Additionally,  one can choose weight data $\bm 
\tau\,=\,(\tau_1,\,\cdots, \,\tau_n)$ in the Weyl alcoves, or equivalently weights 
$\vec{\lambda}\,=\,(\lambda_1,\,\cdots,\, \lambda_n)$, and use them to
define  a suitable notion of 
stability and semistablity. The corresponding moduli spaces $M^{par,ss}_{G, \bm 
\tau}(C,\bm p)$ can, in turn, be understood as the space of representations of the fundamental 
group  of the corresponding punctured surface
$C\setminus \{\bm p\}$, where the loops around the marked points go 
to fixed conjugacy classes determined by $\bm\tau$
\cite{MehtaSeshadri, SeshadriI}. This generalizes the
classical results of Narasimhan-Seshadri \cite{NarSesh} and Ramanathan
\cite{Ramanathan:75}, proved in the non-parabolic case.

The moduli space $M_{G, \bm \tau}^{par,ss}(C,\bm p)$ is equipped with a natural ample 
{\em determinant of cohomology} line bundle $\operatorname{Det}_{par, \phi}(\bm \tau)$ 
associated to a choice of faithful linear representation $\phi$ of $G$. This generalizes the 
theta line bundle on the Jacobian variety $J(C)$. Therefore, the global sections of this line bundle
on $M_{G, \bm \tau}^{par,ss}(C,\bm p)$
can thus be thought of as a nonabelian generalization of the classical theta functions.
We refer the reader to Section \ref{sec:parabolicmoduli} for more 
details on the constructions of the moduli space
and the parabolic determinant line bundle on it.
	
Via  the uniformization theorems of Harder
and Drinfeld-Simpson \cite{DS, Harder},   moduli spaces of parabolic
bundles also have an ad\`elic description that directly connects to the representation theory 
of affine Lie algebras via the work of several authors
(see \cite{BeauvilleLaszlo:94, Faltings:94, KNR:94, LaszloSorger:97,
Pauly:96, Sorger99}). Using this,
 the corresponding moduli stack of 
principal $G$-bundles and its parabolic analog $\mathcal{P}ar_G(C,\bm p, \bm \tau)$ can 
be expressed as a double quotient
$$\mathcal{P}ar_G(C,\bm p,\bm \tau)\,=\,G(\Gamma(C
,\mathcal{O}_C(*\bm p))\backslash
\prod_{i=1}^n G(\mathbb{C}((\xi_i)))/\mathcal{P}_{i}\ , $$ 
where 
the $\mathcal{P}_i$ are parahoric subgroups of $G(\mathbb{C}[[\xi_i]])$ determined by the 
weights $\tau_i$. The weights also determine a  homogeneous
$G(\Gamma(C,\,\mathcal{O}_C(*\bm p)))$-equivariant line bundle $\mathscr{L}_{\vec{\lambda}}$ on 
$\mathcal{P}ar_G(C,\bm p,\bm \tau)$. The line bundle $\operatorname{Det}^{\otimes a}_{par, 
\phi}(\bm \tau)$ coincides with $\mathscr{L}_{\vec{\lambda}}$, where $a$ is a rational number 
determined by the Dynkin index of the representation $\phi$.
Generalizations (see Kumar \cite{kumarbook}, Mathieu \cite{ Mathieu}) of the Borel-Weil theorems (see
\eqref{eqn:pizzatheorem}) for affine flag varieties $G(\mathbb{C}((\xi)))/\mathcal{P}_i$, 
coupled with the ad\`elic description, give a canonical isomorphism (see 
\eqref{eqn:impprojisomorphism}) with conformal blocks
$$\mathcal{V}^{\dagger}_{\vec{\lambda}}(C, \bm p, \bm \xi, \frak{g}, \ell)
\,\cong\, H^0(M^{par, ss}_{G,\bm \tau}(C, \bm p),\, \operatorname{Det}^{\otimes a}_{par, \phi}(\bm \tau)).$$
This isomorphism can be reinterpreted as the Chern-Simon/WZW correspondence. More details 
are given in Section \ref{sec:uniformization}.

Using differential geometric methods, Hitchin, \cite{Hitchin:90}, generalizes a construction 
of Mumford-Welters \cite{welters} to obtain a flat projective connection on
the Friedan-Shenker  
bundle with fibers $H^0(M_G(C),\, \operatorname{Det}^{\otimes \ell})$, from
the viewpoint of {\em 
geometric quantization } in the sense of Kostant-Souriau. This connection also appears in 
Witten's \cite{Witten89} interpretation of Jones polynomial link invariants as 
3-manifold invariants. Hitchin's construction was reinterpreted by van Geemen-de Jong 
\cite{VanGeemenDeJong:98}  sheaf theoretically in terms of the existence of
a ``heat operator", which in the relative setting
is a differential operator that is a combination
of a  first order operator with one that is second order on the fibers (see Section \ref{sec:TDOHitchindejong}).
We recall the details of the general methods of 
Hitchin-van Geemen-de Jong \cite{VanGeemenDeJong:98} in Section \ref{sec:TDOHitchindejong}. 
We also refer to the several complementary approaches of Andersen \cite{Andersen:12}, Axelrod-Witten-della Pietra 
\cite{ADW}, Baier-Bolognesi-Martens-Pauly \cite{BBMP20}, Faltings \cite{Faltings:93}, 
Ginzburg \cite{Ginzburg},  Ran \cite{Ran}, Ramadas \cite{RamaHitchin},  Sun-Tsai \cite{ST} and for 
generalizations to reductive groups, Belkale \cite{Belkale:09}.

In \cite{Laszlo}, Laszlo showed that the connection constructed by Hitchin and the one 
in \cite{TUY:89} coincide under the natural identification
of $H^0(M_{G}(C),\, \operatorname{Det}^{\otimes \ell})$ with 
$\mathcal{V}_{0}^{\dagger}(C,\frak{g},\ell)$. A similar result for twisted Spin groups was 
also proved by Mukhopadhyay-Wentworth \cite{MW}. The following questions are natural in the 
context of parabolic moduli spaces:
\begin{enumerate}
\item Is there a projective heat operator (see Section \ref{sec:TDOHitchindejong} and 
Definition \ref{def:heatoperator}) on the line bundle $\operatorname{Det}^{\otimes 
a}_{par, \phi}(\bm \tau))$ that induces a flat projective connection on the vector bundle 
over $\mathcal{M}_{g,n}$ with fibers $H^0(M^{par, ss}_{G,\bm \tau}(C, \bm p), 
\,\operatorname{Det}^{\otimes a}_{par, \phi}(\bm \tau))$?

\item If such a connection exists, is the identification of conformal blocks with nonabelian 
parabolic theta functions flat with respect to this connection and the WZW/TUY 
connection?
\end{enumerate}

 For $g(C)\,\geq\, 2$, Scheinost-Schottenloher,
\cite{SS}, constructed a parabolic Hitchin connection for $G\,=\,\SL_r$ under the assumption 
that the canonical bundle of $M^{par,ss}_{\SL_r,\bm \tau}(C,\bm p)$ admits a square-root. 
Bjerre \cite{BjerreThesis} removed the ``restriction" in \cite{SS} for $G\,=\,\SL_r$ 
by working on a different parabolic moduli space with full flags. In both 
\cite{BjerreThesis, SS}, the authors construct a connection on the push-forward 
``metaplectically corrected" line bundles of the form $\operatorname{Det}^{\otimes a}_{par, 
\phi}(\bm \tau) \otimes  K_{M^{par,ss}_{\SL_r,\bm \tau}(C,\bm p)}^{1/2}$. We also refer the 
reader to Remark \ref{rem:nonample}. In \cite{BMW1},
we constructed a projective heat operator on 
$\operatorname{Det}^{\otimes a}_{par, \phi}(\bm \tau)$ in general. 
This was produced  from a candidate parabolic Hitchin 
symbol (see \eqref{eqn:symbolonbottom}) satisfying a Hitchin-van Geemen-de Jong type 
equation (see \eqref{eqn:fundamental}){}\footnote{Subsequent to the submission of our paper \cite{BMW1}, in
May 2023 a draft of  the thesis of Zakaria Ouaras appeared in which the author proves the existence
of a unique flat projective connection in the case of moduli spaces of parabolic vector bundles with arbitrary fixed determinant and genus $g \geq 2$}
${}^{,}$ \footnote{We have been informed \cite{Andersen:personal}
that in the case of genus
zero, $\SL_2$, and equal weights $\lambda$ sufficiently small so that
conformal blocks are invariants, the Hitchin
connection constructed in \cite{Andersen:12} agrees with the KZ equation.}.

The following result answers the  above question (2)
and thus generalizes
the result of Laszlo, proved in the non-parabolic case.

\begin{theorem}[{\sc Main Theorem}] \label{thm:maintheorem}
Let $S$ parametrize a smooth family of $n$-pointed curves.
Let $\pi_e\,:\, M^{par,rs}_{G,\bm \tau}
\,\to\, S$ be the relative moduli space parametrizing regularly stable parabolic $G$
bundles, i.e., stable parabolic bundles with minimal automorphisms. The natural isomorphism
$$\mathbb{P}\mathbb{V}^{\dagger}_{\vec{\lambda}}(\frak{g},\ell)\,\isorightarrow\,
\mathbb{P}\pi_{e*}\operatorname{Det}_{par,\phi}^{\otimes a}(\bm \tau)\ , $$
    constructed via the uniformization theorem,
between the projectivizations of the bundles of conformal blocks and nonabelian
parabolic theta functions,  is flat for the
WZW/TUY connection on $\mathbb{P}\mathbb{V}^{\dagger}_{\vec{\lambda}}(\frak{g},\ell)$
and the parabolic Hitchin connection on $\mathbb{P}\pi_{e\ast}
\operatorname{Det}_{par,\phi}^{\otimes a}(\bm \tau)$.
\end{theorem}

We are guided by  a fundamental observation
 that if an algebraic group $G$ acts on a smooth variety $X$ and $\mathscr{L}$ 
 is a $G$-equivariant  line bundle on $X$, then the map induced by 
 the Beilinson-Bernstein localization functor $\operatorname{Loc}:
 \mathcal{U}\mathfrak{g} \rightarrow \Gamma(X, \mathcal{D}(\mathscr{L}))$
 is a quantum analog of the moment map for the $G$ action on $X$, and the
 corresponding graded map $\Sym \mathfrak{g} \rightarrow \Gamma (X, gr
 (\mathcal{D}(\mathscr{L})))$ is dual to the moment map. Hence, it is
 ``independent" of the line bundle $\mathscr{L}$. Thus, an essential point in the proof of this theorem  is  
the fact that the  symbols of the Sugawara operators coming from affine Lie algebras do 
not depend on the highest weights. This is checked via a direct
calculation generalizing the non-parabolic counterparts in the works on Laszlo \cite{Laszlo} and Tsuchiya-Ueno-Yamada \cite{TUY:89}.  The counterpart of this statement on the  moduli of parabolic
bundles side for the parabolic Hitchin symbol is therefore the crux of the
argument. This is carried out in Proposition \ref{prop:simple} using Corollary \ref{cor:mostimportant}. These are the key new features/differences  in  the proof of
Theorem \ref{thm:maintheorem} to that in non-parabolic case considered by Laszlo \cite{Laszlo}.

\paragraph{\bf Application}
Now we discuss  an application of the parabolic generalization of Theorem \ref{thm:maintheorem} 
by giving a geometric reconstruction of the Knizhnik-Zamolodchikov (KZ) equation. 
Let us now focus on the genus zero case. Since $\mathbb{P}^1$ has a global coordinate and a 
global meromorphic two form on $\mathbb{P}^1\times \mathbb{P}^1$ with second order poles 
along the diagonal, the WZW/TUY connection gives a flat (honest) connection on the bundle of 
conformal blocks. The equations for the flat sections are known as
\emph{Knizhnik-Zamolodchikov  equations} \cite{KZ}. Thus, the KZ equations constitute a system of first order differential 
equations, arising from the conformal Ward identities,
that determines $n$-point correlation functions in the Wess-Zumino-Witten-Novikov 
model of two dimensional conformal field theory. The KZ equations have remarkable 
realizations in many other areas. For example higher dimensional generalizations of 
hypergeometric functions are known to be solutions of these equations \cite{SVCoh}. The KZ 
equations can also be regarded as quantizations of the isomonodromy problem for 
differential equations of Fuchsian type \cite{Reshetikhin}. The Kohno-Drinfeld 
\cite{Drinfeld, Kohno} theorem relates the monodromy representation of the braid group 
induced by the KZ connection with solutions of the Yang-Baxter equation.

Here, we consider the KZ equations as equations for flat sections of the trivial vector 
bundle $\mathbb{A}_{\vec{\lambda}}$ over the configuration space $X_n$ of $n$-points in 
$\mathbb{C}$ with fibers 
$$A_{\vec{\lambda}}\,:=\,\operatorname{Hom}_{\frak{g}}(V_{\lambda_1}\otimes \cdots \otimes 
V_{\lambda_n},\,\mathbb{C})\, .$$ We restrict the projective heat operator constructed in 
\cite{BMW1} to the open substack $\mathcal{P}ar^c_{G}(\mathbb{P}^1,\bm p,\vec{\lambda})$ of 
quasiparabolic bundles in case of genus zero, where the underlying
principal $G$-bundle is trivial. This turns 
out to be the quotient stack
$$\mathcal{P}ar^c_{G}(\mathbb{P}^1,\bm p,\vec{\lambda})\,=\, 
[ \left(G/P_{\lambda_1}\times
\cdots \times G/P_{\lambda_n}\right)/G], $$ where $P_{\lambda_i}$ are parabolic subgroup determined by $\lambda_i$ 
and the global sections of the homogeneous line bundle $\mathscr{L}_{\vec{\lambda}}$
are just the invariants $A_{\vec{\lambda}}$. Thus, we 
obtain a flat connection on the vector bundle $\mathbb{A}_{\vec{\lambda}}$
over $X_n$. Finally using Theorem \ref{thm:maintheorem}, we identify this
connection with the KZ connection. 
This gives an alternative geometric construction of the KZ equations. We refer the reader to Section \ref{sec:geometrizationofKZ} and Corollary \ref{cor:geoKZ} for more details and precise statements.

The outline of the paper is as follows.
In Sections \ref{sec:conformalblocksbasic} and \ref{sec:connection} we
review in some detail the construction of the WZW/TUY connection in the
parabolic setting.  In Section
\ref{sec:TDOHitchindejong}, we review the construction from \cite{BMW1} of the Hitchin
connection in the parabolic setting. This involves the metaplectic
correction in a central way.
An important step in this section is the re-expression of one of
the ``controlling equations'' of van Geemen-de Jong for the existence of a
projective heat operator on elements of the rational Picard group (see
Section \ref{sec:metaplectic}).  Finally, in Section \ref{sec:parheatoperator} we state the
fundamental result, Theorem \ref{thm:indepdence}, which provides a
simplification of the expression for parabolic Hitchin symbol. 
This 
 is the geometric reflection of the aforementioned  fact that the Sugawara operators 
 do not depend on the highest weights. 
 
Finally, in Section \ref{sec:sugawara}, using this result we directly
relate the symbols of the Sugawara tensor
and the parabolic Hitchin connection, thus proving Theorem \ref{thm:maintheorem} (see
Section \ref{sec:mainproof}). In the last Section \ref{sec:geometrizationofKZ} we
elaborate on the special case of genus zero curves. In Appendix \ref{sec:moduli}, we define the line bundles over
moduli spaces of parabolic $G$-bundles whose sections give rise to the
Friedan-Shenker bundles. We also relate these line bundles to the
determinant of cohomology in the relative setting.

\medskip 
\paragraph{ \bf Acknowledgments.}\, We thank Prakash Belkale for suggesting the question of 
constructing the Knizhnik-Zamolodchikov connection geometrically. S. M. would like to thank 
Arvind Nair for useful discussions.

	\section{Conformal Blocks}\label{sec:conformalblocksbasic}

In this section we recall the basic notions of conformal blocks, following Tsuchiya-Ueno-Yamada \cite{TUY:89}. Let 
$\mathfrak{g}$ be a complex simple Lie algebra and $\mathfrak{h}\, \subset\,\mathfrak{g}$ a Cartan subalgebra. Let $\Delta$ be a
system of 
roots and $\kappa_{\mathfrak{g}}$ the  Cartan-Killing form, normalized so that 
$\kappa_{\mathfrak{g}}(\theta_{\mathfrak{g}},\,\theta_{\mathfrak{g}})\,=\,2$ for the longest root $\theta_{\mathfrak{g}}$.
We let $\nu_{\gfrak} : \hfrak^\ast\isorightarrow\hfrak$
denote the isomorphism induced by $\kappa_\gfrak$.

\subsection{Affine Lie algebras and integrable modules}\label{se2.1}

Let $\xi$ be a formal parameter. The affine Lie algebra
$$\widehat{\mathfrak{g}}\,:=\,\mathfrak{g}\otimes \mathbb{C}((\xi))\oplus
\mathbb{C}\cdot c$$ 
 is a central extension of the loop algebra $\mathbb{C}((\xi))$ by $c$. 
The Lie bracket operation on $\widehat{\mathfrak{g}}$ is given by the formula 
\begin{eqnarray}
[X\otimes f,\, Y\otimes g]&:= [X,\,Y]\otimes fg + \kappa_{\mathfrak{g}}(X,\,Y)
\operatorname{Res}_{t=0}( gdf)\cdot c,
\end{eqnarray}
where $X,\ Y\in \frak{g}$ and $f,\ g$ are elements of $\mathbb{C}((\xi))$. 

We now briefly recall the basic objects in the representation theory of $\widehat{\mathfrak{g}}$.
It is well-known that the finite dimensional $\mathfrak{g}$ modules are parametrized by the subset
$P_+(\mathfrak{g})\,\subset\, \mathfrak{h}^*$ consisting of dominant integral weights. The representation corresponding to a weight $\lambda$ will be denoted by $V_{\lambda}$. 
 Let $\ell\,>\,0$
be a positive integer, and consider the set 
$P_{\ell}(\mathfrak{g})\,:=\,\{ \lambda \,\in\,
P_+(\mathfrak{g})\,\,\mid\,\,
\kappa_{\mathfrak{g}}(\lambda,\theta_{\mathfrak{g}})\,\leq\, \ell \}.$
The highest weight irreducible integrable representations of
$\widehat{\mathfrak{g}}$ at level $\ell$ are classified by the
set $P_{\ell}(\mathfrak{g})$ defined above. The
$\widehat{\mathfrak{g}}$-module corresponding to $\lambda$ will be denoted by
$\mathcal{H}_{\lambda}$.
\subsection{Sheaf of Conformal Blocks}

Integrable representations of affine Lie algebras were used by Tsuchiya-Ueno-Yamada 
\cite{TUY:89} and Tsuchiya-Kanie \cite{TK88} to define conformal blocks. In this paper, we 
will restrict ourselves to conformal blocks associated to smooth curves. Let $\pi\,:\, 
\mathcal{C}\,\to\, S$ be a family of smooth projective curves, and let ${\bm 
p}\,=\,(p_1,\,\cdots,\, p_n)$ be $n$ non-intersecting sections of the map $\pi$ such that 
the complement $\mathcal{C}\backslash \cup_{i=1}^n p_i(S)$ is an affine scheme.

Consider formal coordinates $\xi_1,\,\cdots,\, \xi_n$ around
the sections $\bm {p}\,=\,(p_1,\,\cdots,\, p_n)$ giving
isomorphisms $ \lim_{m}\mathcal{O}_{\mathcal{C}}/\mathcal{I}^m_{p_i}\,\cong\, \mathcal{O}_{S}[[\xi_i]]$, where
$\mathcal{I}_{p_i}$ is the ideal sheaf of the section $p_i$. 
Let $\mathfrak{g}$ be a simple Lie algebra. As before, we
get a sheaf of $\mathcal{O}_{S}$ Lie algebras defined by 
\begin{equation}\label{z1}
\widehat{\mathfrak{g}}_n(S)\,:=\,\mathfrak{g}\otimes \left( \bigoplus_{i=1}^n
\mathcal{O}_{S}((\xi_i))\right)\oplus \mathcal{O}_{S}\cdot c
\end{equation}
The Lie algebra $\widehat{\mathfrak{g}}_n(S)$ in \eqref{z1} contains a natural subsheaf of Lie algebras
$\mathfrak{g}\otimes_{\mathbb{C}}\pi_*\mathcal{O}_{\mathcal{C}}(*D)$, where $D\,=\,\sum_{i=1}^np_i(S)$. That
it satisfies the Lie algebra condition is actually guaranteed by the residue theorem.

For any choice of $n$-tuple of weights $\vec{\lambda}\,=\,(\lambda_1,\,\cdots, \,\lambda_n)$ 
in $P_{\ell}(\mathfrak{g})$, consider the $\widehat{\mathfrak{g}}_n(S)$-module (hence it is 
also a $\mathfrak{g}\otimes_{\mathbb{C}}\pi_*\mathcal{O}_{\mathcal{C}}(*D)$-module)
$$\mathcal{H}_{\vec{\lambda}}(S)\,:=\, \mathcal{H}_{\lambda_1}\otimes
\mathcal{H}_{\lambda_2}\otimes\dots\otimes \mathcal{H}_{\lambda_n}\otimes_{\mathbb{C}} \mathcal{O}_S.$$

\begin{definition}
	The sheaf of covacua $\mathcal{V}_{ \vec{\lambda}}(\frak{g}, \mathcal{C}/S,\, \bm{p},\,\ell)$ at level $\ell$ is defined to be the
largest quotient of $\mathcal{H}_{\vec{\lambda}}(S)$ on which $\mathfrak{g}\otimes_{\mathbb{C}}\pi_*\mathcal{O}_{\mathcal{C}}(*D)$
acts trivially. The sheaf of conformal blocks $\mathcal{V}^{\dagger}_{ \vec{\lambda}}(\frak{g},\mathcal{C}/S, \bm{p},\ell)$ is
defined to be the $\mathcal{O}_S$ dual of the sheaf of covacua.
\end{definition}

Since the above definition uses only the fact that the formal coordinates identify the completed local ring with the Laurent series
ring, the definition of sheaf of covacua and the sheaf of conformal blocks can be
extended in a straightforward manner to families of nodal curves
with Deligne-Mumford stability property.
	
\subsection{Coordinate free construction}

The sheaf $\widehat{\mathfrak{g}}_n$ and its integrable modules can be
defined without the choice of formal coordinates $\bm \xi$;
we recall this  from \cite{Tsuchimoto}. Let $\pi\,:\, \mathcal{C} \,
\to\, S$ be as above.
Consider the sheaf of formal meromorphic functions on $\mathcal{C}$ with poles along the marked sections 
    $$\mathscr{K}_{\mathcal{C}/p_i}\,\,:=\,\,\lim_{a\to +\infty} \lim_{m\to
    +\infty} \mathcal{O}(ap_i(S))/\mathcal{I}^{m}_{p_i}.$$
	Using this, the {\em coordinate free affine Lie algebra} is defined as follows:
\begin{equation}\label{eqn:coordinatefree}
\widehat{\mathfrak{g}}(\mathcal{C}/p_i)\,:=\,\mathfrak{g}\otimes \mathscr{K}_{\mathcal{C}/p_i}
\oplus \mathcal{O}_S\cdot c. 
\end{equation}
Its $n$-pointed analog is defined to be $$\widehat{\mathfrak{g}}_n(\mathcal{C}/S)\,:=\, \left(\mathfrak{g}\otimes
\bigoplus_{i=1}^n\mathscr{K}_{\mathcal{C}/p_i}\right) \oplus \mathcal{O}_S\cdot c.$$
The Lie bracket operation is defined as in the previous section. 

The Verma module $\mathcal{M}_{\lambda}(\mathcal{C}/S)$ (resp.the coordinate free highest 
weight integrable module $\mathbb{H}_{\lambda}(\mathcal{C}/S)$) can be defined similarly by 
inducing representations (and taking quotients) using a parabolic subalgebra 
$$\widehat{\mathfrak{p} 
}_{p_i}\,:=\, \mathfrak{g} \otimes \widehat{\mathcal{O}}_{\mathcal{C}/p_i}(S)\oplus 
\mathcal{O}_S\cdot c$$ on the finite dimensional module $V_{\lambda}$ via evaluation. More precisely $\mathcal{M}_{\lambda_i}(\mathcal{C}/S)=\operatorname{Ind}_{\widehat{\mathfrak{p}}_{p_i}}^{\widehat{\mathfrak{g}}(\mathcal{C}/p_i)}V_{\lambda_i}$
refer the reader to \cite{Tsuchimoto} for more details.

\begin{definition}
The coordinate free sheaf of covacua $\mathbb{V}_{\vec{\lambda}}(\mathfrak{g},\, \mathcal{C}/S, \,\ell)$ is defined to be the sheaf of
coinvariants $\mathbb{H}_{\vec{\lambda}}(\mathcal{C}/S)/ \left(\mathfrak{g}\otimes_{\mathbb{C}}
\pi_*\mathcal{O}_{\mathcal{C}}(*D)\mathbb{H}_{\vec{\lambda}}(\mathcal{C}/S)\right)$. As before, the sheaf of coordinate free conformal
blocks $\mathbb{V}_{\vec{\lambda}}^{\dagger}(\mathfrak{g},\,\mathcal{C}/S,\, \ell)$ is defined to be
the $\mathcal{O}_{S}$-module dual of $\mathbb{V}_{\vec{\lambda}}(\mathfrak{g},\, \mathcal{C}/S,\, \ell)$.
\end{definition}

Observe that a choice of formal coordinates around the marked points
actually induces isomorphisms between $\widehat{\mathfrak{g}}(\mathcal{C}/p_i)$ and the sheaf $\widehat{\mathfrak{g}}_{\xi_i}
\,:=\,\mathfrak{g} \otimes \mathcal{O}_{S}((\xi_i))\oplus \mathcal{O}_S\cdot c$. 
This identifies the coordinate free conformal blocks and the sheaf
of covacua with those obtained by a choice of coordinates. 

We now recall some important properties of the sheaf of conformal blocks. The reader is
referred to \cite{TUY:89}, \cite{Tsuchimoto}, \cite{Fakhruddin:12}, and \cite{Sorger} for proofs and 
further exposition.

\begin{theorem}\label{th1}
Let $\mathfrak{g}$ be a simple Lie algebra and $\ell\,>\,0$ a positive integer. Then the following statements hold: 
	\begin{enumerate}
\item The sheaf of conformal blocks $\mathbb{V}_{\vec{\lambda}}^{\dagger}(\mathfrak{g},\mathcal{C}/S,\ell)$ carries a flat projective
connection and hence it is locally free \cite{Tsuchimoto}.The Verlinde formula, \cite{Faltings:94, Teleman}, computes
the rank of the vector bundle $\mathbb{V}_{\vec{\lambda}}^{\dagger}(\mathfrak{g},\,\ell)$.

\item Let $\overline{\mathcal{M}'}_{g,n}$ be the moduli stack of stable curves along with a choice of 
formal coordinates around the marked points. Then the sheaf $\mathcal{V}_{\vec{\lambda}}^{\dagger}(\mathfrak{g},\, 
\mathcal{C}/S,\,\ell)$ gives a vector bundle with a flat projective connection with logarithmic 
singularities along the boundary.

\item The sheaf of conformal blocks descends to a vector bundle 
$\mathbb{V}_{\vec{\lambda}}^{\dagger}(\mathfrak{g},\ell)$ with a flat projective connection on the moduli stack 
$\mathcal{M}_{g,n}$. Moreover, the vector bundle $\mathbb{V}_{\vec{\lambda}}^{\dagger}(\mathfrak{g},\ell)$ 
extends to a vector bundle over the Deligne-Mumford compactification $\overline{\mathcal{M}}_{g,n}$, and the 
projective connection extends to a projective connection with logarithmic
            singularities along the boundary.

\item The natural forgetful map form $F\,:\,\overline{\mathcal{M}}'_{g,n}\,\to\, 
\overline{\mathcal{M}}_{g,n}$ identifies the vector bundle 
$\mathcal{V}_{\vec{\lambda}}^{\dagger}(\mathfrak{g},\,\ell)$ with the pull-back 
$F^*\mathbb{V}_{\vec{\lambda}}^{\dagger}(\mathfrak{g},\,\ell)$.

\end{enumerate}
\end{theorem}

We will refer to the connections in Theorem \ref{th1} as the {\em TUY/WZW} connections;
their construction is recalled in Section \ref{sec:connection}.

\begin{remark}
There are several other important properties --- e.g. ``propagation of vacua"
and ``factorization theorems" ---
exhibited by conformal blocks. We refer the reader to \cite{TUY:89} for more details.
	
We also point out that  $\mathcal{V}_{\vec{\lambda}}^{\dagger}(\mathfrak{g},\,\ell)$ and 
$\mathbb{V}_{\vec{\lambda}}^{\dagger}(\mathfrak{g},\,\ell)$ do not agree as twisted $\mathcal{D}$-modules.
This issue does not appear in \cite{Laszlo}, as the weights there are zero.
We refer 
the reader to \cite{Tsuchimoto} for the computation of the Atiyah algebra of 
$\mathbb{V}_{\vec{\lambda}}^{\dagger}(\mathfrak{g},\,\ell)$. 
\end{remark}
	
\section{Energy momentum tensor and the Sugawara construction}\label{sec:connection}

In this section, following the discussion in \cite{TUY:89} we recall the definition of the Sugawara tensor, which will be used in defining the 
connections on the sheaf of covacua and conformal blocks.

For any $X\,\in\, \mathfrak{g}$, the element $X\otimes \xi^m\,\in\,
\widehat\gfrak$ will be denoted by $X(m)$.
The energy momentum tensor $T(z)$ at level $\ell$ is defined by the formula 
\begin{equation}\label{e1}
T(z)\,\,=\,\,\frac{1}{2(\ell+h^{\vee}(\mathfrak{g}))}\sum_{a=1}^{\dim \mathfrak{g}} \norder {J^a(z)J^a(z)}\, , 
\end{equation}
where $\norder{ \ }$ is the normal ordering (cf.\ \cite[p.\ 467]{TUY:89}),
$h^{\vee}(\mathfrak{g})$ is the dual Coxeter number of $\gfrak$,
and $\{J^1,\,\cdots,\, J^{\dim \mathfrak{g}}\}$ is an orthonormal
basis of $\mathfrak{g}$ with respect to the normalized Cartan-Killing form. Also, define
$$X(z)\,\,:=\,\,\sum_{n\in \mathbb{Z}} X(n)z^{-n-1}$$
for any element $X \,\in\, \mathfrak{g}$. 
The $n$-th Virasoro operator $L_n$ is defined by the formula 
(see \cite{KacWakimoto:88}):
\begin{equation}\label{e2}
L_n\,\,:=\,\,\frac{1}{2(\ell+h^{\vee}(\mathfrak{g}))}\sum_{m\in \mathbb{Z}}
\sum_{a=1}^{\dim \mathfrak{g}} \norder {J^a(m)J^a(n-m)}\, .
\end{equation}
The operators $L_n$ act on the module  $\mathcal{H}_{\lambda}$ defined in
Section \ref{se2.1}.

For $X \,\in\, \mathfrak{g}$, $f(z)\,\in\, \mathbb{C}((z))$ and $\underline{\ell}\,=\,\ell(z)\frac{d}{dz}$,
define (as in \cite{TUY:89})
\begin{eqnarray}
X[f]\,:= \,\operatorname{Res}_{z=0}\left( X(z)f(z)\right)dz,\nonumber \ \ \
T[\underline{\ell}]\,:=\,\operatorname{Res}_{z=0}\left(T(z)\ell(z)\right)dz. \label{tl}
\end{eqnarray}

\subsection{Construction of the WZW/TUY connection}\label{se3.1}

Let $C$ be an irreducible smooth projective curve with $n$-marked points
${\bm p}\,=\,(p_1,\,\cdots,\,p_n)$.
For every $1\, \leq\, i\,\leq\, n$, we choose a formal coordinate
$\xi_i$ around the marked point $p_i$ on the curve $C$. 
Following 
\cite{TUY:89}, let us briefly recall the construction of a flat projective connection.

Let $\pi\,:\, \mathcal{C}\,\to\, S$ be a versal 
family of $n$-pointed stable curves equipped with $n$ non-intersecting sections $p_i\,:\, S
\,\to\, \mathcal{C}$. Let $D\,=\,\sum_{i=1}^n p_i(S)$ be the
corresponding divisor
on $\mathcal{C}$. We have a short exact sequence of sheaves 
\begin{equation}\label{eqn:KS1}
0\,\longrightarrow\,\pi_{*}\mathcal{T}_{\mathcal{C}/S}(*D)\,\longrightarrow\,\pi_* \mathcal{T}_{\mathcal{C},\pi}(*D)
\,\longrightarrow\, \mathcal{T}_S\,\longrightarrow\, 0
\end{equation}
on $S$. On the other hand, we also have the short exact sequence
\begin{equation}\label{eqn:KS2}
0\,\longrightarrow\, \pi_* \mathcal{T}_{\mathcal{C}/S}(*D)\,\longrightarrow\,
\bigoplus_{i=1}^n \mathcal{O}_S[\xi_i^{-1}]\frac{d}{d\xi_i}\,\longrightarrow\,
R^1\pi_* \mathcal{T}_{\mathcal{C}/S}(-D)\,\longrightarrow\, 0.
\end{equation}
 obtained from pushing forward the following short exact sequence
 \begin{equation}
 	0\, \longrightarrow\, \mathcal{T}_{\mathcal{C}/S}(-D) \longrightarrow \mathcal{T}_{\mathcal{C}/S}(*D) \longrightarrow \bigoplus_{i=1}^n \mathcal{O}_S[\xi_i^{-1}]\frac{d}{d\xi_i }\longrightarrow 0.
 \end{equation}

Since the family of $n$-pointed curves $\mathcal{C}$ is versal, we have the Kodaira-Spencer isomorphism
$\mathcal{T}_{S} \,  \isorightarrow\, R^1\pi_*\mathcal{T}_{\mathcal{C}/S}(-D)$.
Combining \eqref{eqn:KS1} and \eqref{eqn:KS2}, the following commutative diagram of
homomorphisms is obtained:
	\begin{equation}
	\begin{tikzcd}
	0 \arrow[r]& \pi_{*} \mathcal{T}_{\mathcal{C}/S}(*D)\arrow[r] \arrow[d,equal]& \pi_*\mathcal{T}_{\mathcal{C},\pi}(*D) \arrow[r] \arrow[d,"{\bf t}"]& \mathcal{T}_S \arrow[r]\arrow[d,"KS"]& 0\\
0 \arrow[r] &	\pi_* \mathcal{T}_{\mathcal{C}/S}(*D) \arrow[r] & \bigoplus_{i=1}^n\mathcal{O}_{S}[\xi_{i}^{-1}]\frac{d}{d\xi_i} \arrow[r] \arrow[r,"\theta"] & R^1\pi_* \mathcal{T}_{\mathcal{C}/S}(-D) \arrow[r] &0,
	\end{tikzcd}
	\end{equation}
where ${\bf t}$ is the  projection to the polar part of the
Laurent expansion of sections in terms of the given coordinates $\xi_i$ around the divisors
$p_i(S)$. This map ${\bf t}$ is an isomorphism because the family $\mathcal C$ is versal. 

Let $\vec{\ell}\,=\,(l_1,\,\cdots, \,l_n)$ and $\vec{m}\,=\,(m_1,\,\cdots,\, m_n)$ be two
formal vector fields; both $l_i$ and $m_i$ are defined on a formal neighborhood of
$p_i(S)$.
The Lie bracket $[\vec{\ell},\, \vec{m}]_d$  is given by the formula 
\begin{equation}\label{eqn:liebracketformula}
[\vec{\ell},\, \vec{m}]_d\,:=\,[\vec{\ell},\, \vec{m}]_0+ \theta(\vec{\ell})(\vec{m})-\theta(\vec{m})(\vec{\ell}),
\end{equation}
where $[ \ ,\ ]_0$ is the usual Lie bracket of formal vector fields and $\theta(\vec{\ell})$
acts componentwise using the formal
parameters $\xi_i$. Now for any formal vector field $\vec{\ell}$, define the operator $\mathscr{D}(\vec{\ell})$
on $\mathcal{H}_{\vec{\lambda}}$ by the formula 
\begin{equation}\label{dl}
\mathscr{D}(\vec{\ell})(F\otimes | \Phi\rangle )\,\,
    :=\,\,\theta(\vec{\ell})(F)\otimes |\Phi\rangle -F \cdot\left(\sum_{j=1}^n\rho_j(T[\underline{l}_j])\right)|\Phi\rangle,
	\end{equation}
    where $\rho_j$ is the action on $\mathcal{H}_{\lambda_j}$ defined on
    \cite[p.\ 475]{TUY:89}.

\subsection{WZW/TUY connection}

After possibly shrinking $S$, we can find a symmetric bidifferential $\omega$ on
$\mathcal{C}\times_S \mathcal{C}$ with a pole of order two on the diagonal such
that the biresidue is $1$. For any formal vector field ${\vec{\ell}}$ define
$$
a_{\omega}(\vec{\ell})\,:=\,
-\frac{c_{v}}{12}\sum_{i=1}^n \operatorname{Res}_{\xi_i=0}\left( \ell_i(\xi_i)S_{\omega}(\xi_i)d\xi_i\right),
$$
where $S_{\omega}$ is the
projective connection associated to $\omega$,
and  $\displaystyle c_{v}\,=\,\frac{\ell\dim \mathfrak{g}}{\ell+h^{\vee}(\mathfrak{g})}$ is the central charge. 

Now let $\tau$ be a vector field on $S$. Take a lift of $\tau$ to $\pi_{*} \mathcal{T}_{\mathcal{C}/S}(*D)$ and denote it by
$\vec{\ell}$. With the choice of a bidifferential $\omega$ as above, we define the following operator on the sheaf
of conformal blocks: 
\begin{equation}\label{k1}
\nabla_{\tau}^{(\omega)}(\langle \Psi| )\,:=\,\mathscr{D}(\vec{\ell})(\langle \Psi| ) + a_{\omega}(\vec{\ell})(\langle \Psi|).
\end{equation}
We recall the following result (see \cite[Thm.\ 5.3.3]{TUY:89}).

\begin{proposition}
The operator $\nabla^{(\omega)}_{\tau}$ in \eqref{k1} defines a flat projective connection on the sheaf of
conformal blocks $\mathcal{V}_{\vec{\lambda}}^{\dagger}(\mathfrak{g},\,\ell)$. 
\end{proposition}
\begin{remark}
Consider the natural forgetful map $F\,:\, \overline{\mathcal{M}}_{g,n}'\,\to \,
\overline{\mathcal{M}}_{g,n}$ constructed
by forgetting the choice of formal parameters at the $n$-marked points.
Then the natural identification between $F^*\mathbb{V}_{\vec{\lambda}}(\mathfrak{g},\,\ell)$ and
$\mathcal{V}_{\vec{\lambda}}(\mathfrak{g},\,\ell)$ as locally free sheaves intertwines, up to a first order operator,
the pull-back of the coordinate free TUY connection on $F^*\mathbb{V}_{\vec{\lambda}}(\mathfrak{g},\ell)$
and the TUY connection on $\mathcal{V}_{\vec{\lambda}}(\mathfrak{g},\,\ell)$. 
\end{remark}

\section{Projective connections via heat operators}\label{sec:TDOHitchindejong}
\subsection{Heat operators and the Hitchin-van Geemen-de Jong equation}

Let $\pi\,:\, M \,\to\, S$ be a smooth map of smooth varieties, and let $\mathcal{L}$ be a line bundle on $M$. 
The Kodaira-Spencer infinitesimal deformation map is given by:
$$KS_{M/S}: \mathcal{T}_{S}\lra R^1\pi_{*} \mathcal{T}_{M/S}.$$
On the other hand, we have  the
coboundary map $$\mu_{\mathcal{L}}: \pi_*\operatorname{Sym}^2\mathcal{T}_{M/S}
\longrightarrow R^1\pi_* \Tcal_{M/S},$$
occurring in the long exact sequence obtained from the push forward $\pi_{*}$ of the fundamental
short exact sequence of differential operators
\begin{equation*}
0 \lra \mathcal{T}_{M/S} \cong \mathcal{D}_{M/S}^{\leq
	1}(\mathcal{L})/\mathcal{O}_M\lra \mathcal{D}_{M/S}^{\leq
	2}(\mathcal{L})/\mathcal{O}_{M} \stackrel{s_2}{\lra} \operatorname{Sym}^2
\mathcal{T}_{M/S}\lra 0\ ,
\end{equation*}where $s_2$ is the symbol map and $\mathcal{D}^{\leq i}_{M/S}(\mathcal{L})$ is the sheaf of relative differential operators of order atmost $i$. Following \cite{VanGeemenDeJong:98}, consider the sheaf $\mathcal{W}(\mathcal{L})$ defined by
\begin{equation}
\mathcal{W}(\mathcal{L})\,:= \,\mathcal{D}_{M}^{\leq 1}(\mathcal{L})+ \mathcal{D}_{M/S}^{\leq 2}(\mathcal{L}).
\end{equation}
There is a natural short exact sequence 
\begin{equation}\label{q0}
0\,\lra\, \mathcal{D}_{M/S}^{\leq 1}(\mathcal{L})\,\lra\, \mathcal{W}(\mathcal{L})\,
\stackrel{q_0}{\longrightarrow}\,
\pi^* \mathcal{T}_S \oplus \operatorname{Sym}^2 \mathcal{T}_{M/S}\,\lra\, 0.
\end{equation}

\begin{definition}\label{def:heatoperator}
A {\em heat operator} $\mathbb{D}\,:\, \pi^* \mathcal{T}_{S}\,\to\, \mathcal{W}({\mathcal{L}})$ is a section of the natural
projection map $\mathcal{W}(\mathcal{L}) \,\to\, \pi^*\mathcal{T}_S$. A {\em projective heat operator}
is a section of $\mathcal{W}(\mathcal{L})/\mathcal{O}_M\,\to\, \pi^*\mathcal{T}_S$.
\end{definition}

A projective heat operator evidently lifts, locally, to a heat operator. 
Given a homomorphism 
$\rho\,:\, \pi^*\mathcal{T}_S\,\to\,\operatorname{Sym}^2 \mathcal{T}_{M/S}$, 
one can ask whether there is a canonical {\em projective heat operator} $\mathbb{D}\, :\, 
\pi^*\mathcal{T}_S\,\to\, \mathcal{W}(\mathcal{L})/\mathcal{O}_M$ such that 
$q_1\circ \mathbb{D}\,=\, \rho$. The following theorems of Hitchin \cite{Hitchin:90} and 
van Geemen-de Jong \cite{VanGeemenDeJong:98}, answer a fundamental question on existence of projective heat operators. 

\begin{theorem}[{\cite{Hitchin:90, VanGeemenDeJong:98}}]\label{thm:existenceofheat}
	Assume that the following conditions hold: 
	\begin{itemize}
		\item $KS_{M/S} + \mu_{\mathcal{L}}\circ \rho\,=\,0$;
		
		\item $\cup [\mathcal{L}] \,:\,
            \pi_*\mathcal{T}_{M/S} \,\to\,
		R^1\pi_*\mathcal{O}_{M}$ is an isomorphism; 
		
		\item $\pi_*\mathcal{O}_M\,=\,\mathcal{O}_S$.
	\end{itemize}
	Then there exists a unique projective heat operator $\overline{\mathbb{D}}$ lifting
any candidate symbol
$\rho\,:\, \pi^*\mathcal{T}_S \,\to\, \operatorname{Sym}^2\mathcal{T}_{M/S}$. Moreover the coherent sheaf $\pi_*\mathcal{L}$ carries a projective connection.
\end{theorem}

\begin{remark}\label{rem:extntofrac}
We can take $\mathcal{L}$ to be an object in the rational Picard group $\operatorname{Pic}(M)\otimes \mathbb{Q}$. 
All the sheaves that appear in the statement of 
    Theorem \ref{thm:existenceofheat} are well-defined, and the proof 
     in \cite{VanGeemenDeJong:98} does not
require the assumption that $\mathcal{L}$ be a line bundle.
\end{remark}

\subsection{Heat operators and metaplectic quantization}
\label{sec:metaplectic}

We are interested in the case where the Kodaira-Spencer map $KS_{M/S}$, the candidate symbol $\rho\,:\,
\mathcal{T}_S\,\to\, 
\pi_*\operatorname{Sym}^2 \mathcal{T}_{M/S}$ and the class of $\mathcal{L}$
are intertwined by the equation:
\begin{equation}\label{eqn:deformation1}
KS_{M/S}+ \cup[\mathcal{L}]\circ \rho\,=\,0 .
\end{equation}
We refer to \eqref{eqn:deformation1} as the equation {\em controlling the deformations}. 
Recall that the connecting homomorphism
$$\mu_{\mathcal{L}^{{\otimes k}}}\,:\, \pi_*\operatorname{Sym}^2 \mathcal{T}_{M/S}\,\lra\,
R^1\pi_*\mathcal{T}_{M/S}$$
is given by the formula (see \cite{BBMP20})
\begin{equation}\label{ch}
\mu_{\mathcal{L}^{\otimes k}}\,=\,\cup \left([k\cdot\mathcal{L}]-\frac{1}{2}[K_{M/S}]\right).
\end{equation}

\subsubsection{Rewriting the deformation equation}

We will now rewrite \eqref{eqn:deformation1} to produce a projective heat operator 
on $\mathcal{L}$. Throughout this subsection we assume that for any
positive $k \,\in\, \mathbb{Q}$, 
the connecting homomorphism $\mu_{\mathcal{L}^{\otimes k}}$ in \eqref{ch}
is an isomorphism. This condition holds, for example,  in the case of
moduli spaces of parabolic bundles. 

Now
\begin{eqnarray*}
	KS_{M/S} + \cup [\mathcal{L}]\circ \rho && =\, KS_{M/S}+\frac{1}{k}\left(\cup\left(k[\mathcal{L}]-\frac{1}{2}K_{M/S}\right)+ \cup \frac{1}{2}[K_{M/S}]\right)\circ \rho\\
	&&=\,KS_{M/S}+\mu_{\mathcal{L}^{\otimes k}}\circ \left(1+ \mu_{\mathcal{L}^{\otimes k}}^{-1} \circ (\cup \frac{1}{2}[K_{M/S}])\right)\circ \frac{1}{k}\rho\\
	&&=\,KS_{M/S}+\mu_{\mathcal{L}^{\otimes k}}\circ \widetilde{\rho}_k,
\end{eqnarray*}
where $\widetilde{\rho}_k\,=\,\left( 1+ \mu_{\mathcal{L}^{\otimes k}}^{-1}\circ \left(\cup \frac{1}{2}[K_{M/S}]\right)\right)\circ \frac{1}{k}{\rho}$ is the symbol map. 
Again assuming that the conditions of Theorem \ref{thm:existenceofheat} are
satisfied, we get a projective heat operator $\overline{\mathscr{D}}$ on $\mathcal{L}$ with symbol $\widetilde{\rho}_k$ such that the following diagram commutes 
\begin{equation}
\begin{tikzcd}
&\pi_* \mathcal{D}^{\leq 2}_M (\mathcal{L}^{\otimes k})\arrow[d, "symb"]\\
\mathcal{T}_S \arrow[ru,dotted] \arrow[r,"\rho_k"] & \pi_*\operatorname{Sym}^2\mathcal{T}_{M/S}
\end{tikzcd}
\end{equation}
This induces a projective connection on $\pi_*\mathcal{L}^{{\otimes k}}$
for any positive $k\,\in\, \mathbb{Z}$ with symbol $\widetilde{\rho}_k$

\subsubsection{Metaplectic correction d'apr\`es Scheinost-Schottenloher}

We can rewrite the left-hand side of \eqref{eqn:deformation1} as follows:
\begin{eqnarray*}
	KS_{M/S}+ \cup [\mathcal{L}]\circ \rho&&=\,KS_{M/S}+\frac{1}{k}\left(\cup\left([k\cdot \mathcal{L}]+\frac{1}{2}[K_{M/S}]\right)-\cup\frac{1}{2}[K_{M/S}]\right)\circ \rho\\
	&&= \,KS_{M/S}+ \mu_{\mathcal{L}^{{\otimes k}}\otimes K^{\frac{1}{2}}_{M/S}}\circ \rho_k,
\end{eqnarray*}
where $\rho_k\,:=\,\frac{1}{k}\rho$ and $\mathcal{L}^{\otimes k} \otimes K_{M/S}^{1/2}$ is considered as an element
of the rational Picard group. Thus, from \eqref{eqn:deformation1} the following equation
is obtained:
\begin{equation}\label{eqn:deformation2}
KS_{M/S}+ \mu_{\mathcal{L}^{{\otimes k}}\otimes K_{M/S}^{\frac{1}{2}}}\circ \rho_k\,=\,0.
\end{equation}

Assume that the other conditions of the Hitchin-van Geemen-de Jong existence theorem are satisfied.
Then Theorem \ref{thm:existenceofheat} tells us that there exists a unique projective heat operator 
$\widehat{\mathbb{D}}_{k}$ with symbol
$\rho_k$ and a connection on $\pi_*(\mathcal{L}^{k}\otimes K_{M/S}^{\frac{1}{2}})$.  As pointed out
in Remark \ref{rem:extntofrac}, the projective heat operator makes sense even if the square-root of 
$K_{M/S}$ does not exist.

It is easy to see that for any 
candidate symbol, there exists a second order projective operator
$\widehat{\mathbb{D}}$ on $K^{\frac{1}{2}}$ with the same given symbol. However, 
this operator is not a projective heat operator, since the natural projection of it to 
$\pi^*\mathcal{T}_{S}$ is zero. On the other hand, we have a projective heat operator $\overline{\mathbb{D}}$ on $\mathcal{L}^{\otimes k}$ with the same symbol 
$\widetilde{\rho}_k$. The following is then a natural question.

\begin{question}\label{qa}
Using the projective heat operator $\overline{\mathbb{D}}$ and the projective operator $\widehat{\mathbb{D}}$ with symbol
${\rho}_k$, can we construct a projective heat operator $\widetilde{\mathbb{D}}$ on $\mathcal{L}^k\otimes K_{M/S}^{\frac{1}{2}}$?
\end{question}

\begin{remark}
Observe that the equations in Theorem \ref{thm:existenceofheat} imply that there exists at most one heat operator 
provided
$$\mu_{\mathcal{L}^{\otimes k}\otimes K_{M/S}^{\frac{1}{2}}}\,:\, \pi_*\operatorname{Sym}^2 \mathcal{T}_{M/S}
\,\longrightarrow\, R^1\pi_*\mathcal{T}_{M/S}$$ is an isomorphism.
A positive answer to Question \ref{qa}
would immediately imply that the symbol of $\widetilde{\mathbb{D}}$
satisfies the equation for $\rho_{k}$ given  in 
\eqref{eqn:deformation2}. This will provide a necessary relation that the linear maps $\cup[{\mathcal{L}}]$ 
and $\cup[{K_{M/S}}]$ are scalar multiples of each other. This would give
an alternative, more conceptual understanding of Theorem 4.1 in \cite{BMW1} \end{remark}

\section{Parabolic Hitchin symbol as in Biswas-Mukhopadhyay-Wentworth }

In this section we first briefly recall the construction of the Hitchin connection for the 
moduli space of parabolic bundles $M^{par,rs}_{G,\bm{\tau}}$ obtained in
\cite{BMW1}. We refer the reader to Appendices \ref{sec:parabolicmoduli} and \ref{sec:paradet} for a brief review of the moduli stack of parabolic bundles and the parabolic determinant of cohomology line bundles. We will freely use the correspondence between parabolic $G$ bundles on a curve $C$ and equivariant ($\Gamma$-$G$)-bundles on a Galois cover $\widehat{C}$ of the curve $C$ with Galois group $\Gamma$. This is recalled in Appendix \ref{sec:orbcorr}.  

In particular, we focus on the parabolic Hitchin symbol defined in 
the paper \cite{BMW1}. Using restriction to fibers of the Hitchin map, we give a 
simplification of the expression for the symbol that enables us to compare the parabolic 
Hitchin symbol with the symbol of the Sugawara operators as constructed in \cite{TUY:89}. The main result of this section is Proposition \ref{prop:simple}. This is a key new feature and one of the main technical difficulties in the parabolic set-up that does not appear in \cite{Laszlo}.

\subsection{The Hitchin symbol}

In \cite{BMW1} (recalled in Appendix \ref{sec:orbcorr}), we identified the moduli space $M^{par}_{G,\bm 
\tau}$ of parabolic bundles with the moduli space $M^{\bm \tau,ss}_{G}$ of 
$(\Gamma,\,G)$-bundles on a Galois cover $\widehat{C}$ of the curve $C$ via the invariant pushforward 
functor \cite{BalajiBiswasNagaraj, BalajiSeshadri}.

This includes the following identifications:
Let $\mathcal{P}$ be a regularly stable parabolic bundle and $\widehat{\mathcal{P}}$ 
the corresponding $(\Gamma,\,G)$ bundle (see Appendix \ref{sec:orbcorr}). 
Let $\operatorname{Par}(\Pcal)$ (resp.\ $\operatorname{Spar}(\Pcal)$)
denote the bundle of parabolic (resp.\ strictly parabolic) automorphisms of
$\Pcal$. 
Then
\begin{equation}\label{eqn:iden1}
H^0(C,\, \operatorname{Spar}(\Pcal)\otimes K_C(D))\,\cong\, H^0(\widehat{C},\, \ad \mathcal{P}\otimes K_{\widehat{C}})^{\Gamma}
\end{equation}
and
\begin{equation}\label{eqn:iden2}
H^1(C,\, \operatorname{Par}(\Pcal))\,\cong\, H^1(\widehat{C},\, \ad \mathcal{P})^{\Gamma}.
\end{equation}

The Hitchin symbol was defined using the natural map 
\begin{equation}
\label{eqn:oldsymbol}
H^0(\widehat{C},\, \ad \mathcal{P}\otimes K_{\widehat{C}})^{\Gamma}\otimes H^1(\widehat{C},\,
\ad \mathcal{P})^{\Gamma} \xrightarrow{\ \kappa_g\ } H^1(\widehat{C},\, K_{\widehat{C}} )^{\Gamma} \,\cong\, \mathbb{C};
\end{equation}
where the last isomorphism is given by the Serre duality on $\widehat{C}$.
As in \cite[Prop.\ 2.16]{Hitchin:90}, this induces a natural map 
\begin{equation}\label{eqn:symbolontop}
\rho_{sym}\,:\, R^1\pi_{s *}\mathcal{T}_{\mathcal{C}/S}(-D)
\,\longrightarrow \,\pi_{e *}\operatorname{Sym}^2 \mathcal{T}_{M^{\bm \tau, rs}_{G}/S}\, ,
\end{equation}
where $\pi_{e}\,:\, M_{G}^{\bm \tau,rs}\,\to\,S$ and
$\pi_{s}\,: \,\mathcal{C}\,\to\, S$ are the projections.
On the other hand we also have the pairing 
\begin{equation}\label{eqn:naturalcech}
\begin{tikzcd}
H^0(C,\, \operatorname{Spar}(\Pcal)\otimes K_C(D))\otimes H^1(C,\, \operatorname{Par}(\Pcal))\arrow[d]&\\
H^1(C,\,\operatorname{Spar}(\Pcal)(D)\otimes \operatorname{Par}(\Pcal)\otimes K_C)	\arrow[r, "\kappa_{\mathfrak{g}}"]
&H^1(C,\,K_C)\,\cong\, \mathbb{C}.
\end{tikzcd}
\end{equation}
This also induces a map 
\begin{equation}
\label{eqn:symbolonbottom}
\widetilde{\rho}_{sym}\,:\, R^1\pi_{s*}\mathcal{T}_{\mathcal{C}/S}(-D)\,
\longrightarrow\, \pi_{e*} \operatorname{Sym}^2\mathcal{T}_{M_{G,\bm \tau}^{par,rs}/S}.
\end{equation}

The isomorphisms $H^1(C,\,K_C)\,\cong\, \mathbb{C}$ in \eqref{eqn:naturalcech} and 
$H^1(\widehat{C},\,K_{\widehat{C}})\cong\, \mathbb{C}$ are both given by the residue theorem 
and Serre duality, but for different curves. 
Hence, the  map on Hitchin symbols $\rho_{sym}$ and 
$\widetilde{\rho}_{sym}$ do not commute under the identifications given by \eqref{eqn:iden1} and \eqref{eqn:iden2}. However they are related
as follows:

\begin{lemma}\label{lem:normalization}
	Let $\rho_{sym}$ and $\widetilde{\rho}_{sym}$ be as above. Then
	$|\Gamma|\cdot \rho_{sym} \,=\,\widetilde{\rho}_{sym}. $
\end{lemma}

\begin{proof}
This is immediate from the commutativity of
$$
    \begin{tikzcd}
        H^1(C,K_C)\arrow[r,"\sim"]\arrow[d] & H^1(\widehat{C},\, K_{\widehat{C}}
        )^{\Gamma}\arrow[d] \\
        \CBbb\arrow[r,"\times |\Gamma|"] & \CBbb
    \end{tikzcd}
$$
\end{proof}

\subsection{Parabolic Hitchin connection via heat operators}\label{sec:parheatoperator}

Let $\phi \,:\, G\,\to\, \SL(V)$ be a linear representation satisfying the hypothesis 
of Section \ref{sec:paradet}. Let $\mathcal{L}_{\phi}$ be the pullback of the determinant 
of cohomology line bundle to $M^{\bm \tau,ss}_{G}$. Via the identification of parabolic 
$G$-bundles as equivariant bundles and Proposition \ref{prop:identification}, we have 
identified it with the parabolic determinant of cohomology 
$\operatorname{Det}_{par}(\nu_{\mathfrak{sl}(V)}(\phi( \bm \tau)))$. For notational 
convenience, we will drop $\phi$ when denoting the line bundle 
$\operatorname{Det}_{par, \phi}(\mathcal{P}, \bm \tau)$ and 
simply write $\operatorname{Det}_{par}(\mathcal{P}, \bm \tau)$.

In \cite{BMW1}, the authors produced a projective heat operator on the line bundle $\mathcal{L}^{\otimes k}_{\phi}$ whose symbol satisfies the following Hitchin-van Geemen-de Jong  equation:
\begin{equation}
KS_{M^{par,rs}_{G,\bm \tau}/S}+ \mu_{\mathcal{L}^{k}_{\phi}}\circ \left( \frac{1}{m_{\phi}k}\operatorname{Id}+ \mu^{-1}_{\mathcal{L}^{k}_{\phi}}\circ \left(\cup \frac{1}{2m_{\phi}k}
    [K_{M^{par,rs}_{G, \bm \tau}/S}]\right)\right)\circ \rho_{sym}\circ KS_{\mathcal{C}/S}=0.
\end{equation}\label{eqn:fundamental}
Setting $\rho=\rho_{sym}\circ KS_{\mathcal{C}/S}$ gives 
\begin{equation}\label{eqn:f1}
KS_{M^{par,rs}_{G, \bm \tau}/S}+ \mu_{\mathcal{L}^{k}_{\phi}}\circ \left(
    \operatorname{Id}+ \mu^{-1}_{\mathcal{L}^{k}_{\phi}}\circ \left(\cup
    \frac{1}{2}[K_{M^{par,rs}_{G, \bm \tau}/S}]\right)\right)\circ \frac{1}{m_{\phi}\cdot k}\rho=0,
\end{equation} where $k$ is a rational number.

Let $k\,=\,\ell/|\Gamma|$. Using the identification 	$\mathcal{L}_{\phi} \cong 	\left(\operatorname{Det}_{par}(\Pcal,\bm \tau)\right)^{\frac{|\Gamma|}{\ell}}$ in Proposition \ref{prop:identification}, from
\eqref{eqn:f1} we get that 
$$KS_{M^{par,rs}_{G, \bm{\tau}}/S}+ \mu_{\operatorname{Det}_{par}(\Pcal, \bm \tau)\otimes
K^{1/2}_{M^{par,ss}_{G, \bm \tau}/S} }\circ \frac{|\Gamma|}{
m_{\phi}\cdot \ell}\rho\,=\,0.$$
\begin{definition}\label{def:addingadj}
	For any rational number $a$, we will denote the projective heat operator on $\operatorname{Det}^{\otimes a}_{par,\phi}(\Pcal, \bm \tau)$ obtained in \cite{BMW1}
    by $\mathbb{D}(\frak{g}, a\cdot m_{\phi}\cdot\ell)$. 
\end{definition}
The following is one of the main results of \cite[Theorem 4.1]{BMW1}.

\begin{theorem}\label{thm:indepdence}
Let $\mathbb{L}$ be an element of
$\operatorname{Pic}(M^{par,rs}_{G, \bm{\tau}})\otimes \mathbb{Q}$ of
level $a$. Then there is an equality $\cup [\mathbb{L}]\,=\,\cup a [\operatorname{Det}]$ as linear maps
$\pi_{e*}\operatorname{Sym}^2\mathcal{T}_{M^{par,rs}_{G, \bm{\tau}}/S}\,
\to\, R^1\pi_{e*}\mathcal{T}_{M^{par,rs}_{G,\bm{\tau}}/S}$, where $\operatorname{Det}$ is the determinant of cohomology (non-parabolic) line bundle. 
\end{theorem}

\begin{remark}
The above result should be put in the more general context of deformation theory of the 
moduli space of the parabolic bundles as studied in Boden-Yokogawa \cite{BY}, and the 
birational variation of these moduli spaces as the weights vary.
\end{remark}

Since line bundles on $M_{G,\bm \tau}^{par,rs}$ are pull-backs of rational powers of line 
bundles from $M_{\SL(V), \bm \alpha}^{par,s}$ for an 
appropriate choice of  representation 
$(\phi,\,V)$ of $G$, the following is an immediate consequence of Theorem 
\ref{thm:indepdence}.

\begin{corollary}\label{cor:mostimportant}
Let $\widehat{M}_G$ and $M_{G, \bm \tau}^{par,rs}$ be as in Section
\ref{sec:paradet}. Let $(\phi,\,V)$ be a representation of $G$. Then the line bundle $K_{\widehat{M}_G/S}$
restricted to $M_{G, \bm \tau}^{par,rs}$ is of level $-\frac{2h^{\vee}(\mathfrak{g}).|\Gamma|}{m_{\phi}}$ with respect to $\SL(V)$, and hence
$$\cup [K_{M_{G, \bm \tau}^{par,rs}/S}]\,=\,\cup \frac{[{K_{\widehat{M}_G/S}}]}{|\Gamma|}$$ as
linear maps from $\pi_*\operatorname{Sym}^2\mathcal{T}_{M_{G, \bm \tau}^{par,rs}/S}$ to
$R^1\pi_*\mathcal{T}_{M_{G, \bm \tau}^{par,rs}/S}$.
\end{corollary}

\begin{proof}
    On the one hand, we have  the fact from  \cite{KNR:94} that
    $K_{\widehat{M}_{G}/S}$ is
    $-2h^{\vee}(\mathfrak{g})\mathcal{L}_{\phi}$,
    where $\mathcal{L}_{\phi}$ is the ample generator of the Picard group
    of the moduli stack of $\SL(V)$ bundles on $\widehat{C}$, where $\widehat C\to
    C$ is a $\Gamma$-cover. Moreover, $\mathcal{L}_{\phi}$ restricted to
    $M_{G, \bm \tau}^{par,rs}/S$ is of level $|\Gamma|$ (see Appendix \ref{sec:paradet}).
On the other hand,   the canonical bundle $K_{M_{G, \bm
    \tau}^{par,rs}/S}$ has a component that is  $-2h^{\vee}(\mathfrak{g})$ times  the ample
    generator of the Picard group of the moduli stack of $G$ bundles on
    $C$, which in turn is of level $\frac{1}{m_{\phi}}$ with respect to the $\mathfrak{sl}(V)$ embedding of $\mathfrak{g}$. 
    This proves the result.
\end{proof}
Corollary \ref{cor:mostimportant} further simplifies the parabolic Hitchin symbol, as shown by 
the following.

\begin{proposition}\label{prop:simple} Assume that $m_{\phi}\cdot\ell+h^{\vee}(\mathfrak{g})\,\not=\, 0$. Then
\begin{equation}
\frac{1}{m_{\phi}\ell}\left(1+ \mu_{\operatorname{Det}_{par}(\bm \tau)}^{-1}\circ \left(\frac{1}{2}\cup
[K_{M^{par,rs}_{\bm \tau,G}/S}]\right)\right)\circ {\widetilde{\rho}_{sym}}\, =\,\frac{\widetilde{\rho}_{sym}}{m_{\phi}\cdot\ell+h^{\vee}(\mathfrak{g})}.
\end{equation}
\end{proposition}

\begin{proof}Since $\widetilde{\rho}_{sym}$ is invertible, we need to show that 
$$\frac{m_{\phi}\cdot\ell+h^{\vee}(\mathfrak{g})}{m_{\phi}\cdot\ell}\left(1+
\mu_{\operatorname{Det}_{par}(\bm \tau)}^{-1} \left(\frac{1}{2}\cup
[K_{M^{par,rs }_{\vec{\alpha},\SL_r}/S}]\right)\right) \,=\,\operatorname{Id}.$$
So it suffices to prove that
$$\mu_{\operatorname{Det}_{par}}^{-1}(\bm \tau)\circ \left(\frac{1}{2}\cup
[K_{M^{par,rs }_{G, \bm \tau}/S}]\right)\,=\,\left(-1+\frac{m_{\phi}\cdot\ell}{m_{\phi}\cdot\ell+h^{\vee}(\mathfrak{g})}\right)\operatorname{Id}.$$
By multiplying with $\mu_{\operatorname{Det}_{par}(\bm \tau)}$, it suffices to show that
\begin{equation}
\cup [K_{M^{par,rs}_{G, \bm \tau}/S}]\,=\,\left(\frac{-2h^{\vee}(\mathfrak{g})}{m_{\phi}\cdot \ell+h^{\vee}(\mathfrak{g})}\right)\mu_{\operatorname{Det}_{par}(\bm \tau)}.
\end{equation}
Now by \cite{KNR:94, KumarNarasimhan:97}, applied to the moduli space $\widehat{M}_G$ of
principal $G$ bundle on $\widehat{C}$, we get that
\begin{equation}
[K_{\widehat{M}_{G}}]\,=\,-2h^{\vee}(\mathfrak{g}).m_{\phi}.[\mathcal{L}_{\phi}].
\end{equation}
By Corollary \ref{cor:mostimportant},
$$
\cup [K_{M^{par,rs }_{G, \bm \tau }/S}]=\cup
    \frac{-2h^{\vee}(\mathfrak{g})[\mathcal{L}_{\phi}]}{m_{\phi}\cdot|\Gamma|}
=\cup\frac{-2h^{\vee}(\mathfrak{g})}{m_{\phi}\cdot |\Gamma|}\left( \frac{|\Gamma|[\operatorname{Det}_{par}(\bm \tau)]}{\ell}\right).
$$
Rewriting the above equation, we find
\begin{eqnarray*}
\cup m_{\phi}\cdot\ell [K_{M^{par,rs}_{G, \bm \tau}/S}]&\,=\,&\cup \left(\frac{-2h^{\vee}(\mathfrak{g})\cdot m_{\phi}\cdot\ell}{|\Gamma|\cdot m_{\phi}}\right)\frac{|\Gamma|[\operatorname{Det}_{par}(\bm \tau)]}{\ell},\\
\cup \frac{m_{\phi}\cdot \ell}{m_{\phi}\cdot\ell+h^{\vee}(\mathfrak{g})}[K_{M^{par}_{G,
\bm \tau}/S}]&\,=\,&\left(\cup \frac{1}{m_\phi\cdot\ell+h^{\vee}(\mathfrak{g})} \left(-2h^{\vee}(\mathfrak{g})[\operatorname{Det}_{par}]\right)\right),\\
\cup [K_{M^{par,rs}_{G, \bm \tau}/S}]&\,=\,&\left(\cup\frac{-2h^{\vee}(\mathfrak{g})}{m_\phi\cdot\ell+h^{\vee}(\mathfrak{g})}\left( [\operatorname{Det}_{par}(\bm \tau)]-\frac{1}{2}[K_{M^{par,rs}_{\bm \tau, G}/S}]\right)\right),\\
\cup \frac{1}{2}[K_{M^{par}_{\bm \tau,G}/S}]&\,=\,&
    \left(\frac{-h^{\vee}(\mathfrak{g})}{m_{\phi}\cdot
    \ell+h^{\vee}(\mathfrak{g})}\right)\cdot
    \mu_{\operatorname{Det}_{par}(\bm \tau)}\qquad\qquad\text{(from
    \eqref{ch})}\ .
\end{eqnarray*}
This completes the proof.
\end{proof}

\section{Proof of Theorem \ref{thm:maintheorem} }
\label{sec:sugawara}
In this section, we give a proof of the main theorem in this article by comparing the Sugawara tensor and the parabolic analog of the heat operator with a given symbol constructed by the authors in \cite{BMW1}.

Let $\mathcal{P}ar^{rs}_{G}(C,\bm p,\bm \tau)$ be the open substack parametrizing regularly 
stable parabolic bundles of parabolic type $\bm \tau$. Then the natural map
$\mathcal{P}ar_G^{rs}(C,\bm p , \bm \tau)\,\to\, M_{G,\bm \tau}^{par,rs}(C,\bm p)$ is a 
gerbe banded by the center $Z(G)$ of the group $G$.

Similarly, let ${\mathcal{Q}}_{\bm \tau}^{rs}$ be the open ind-subscheme of 
${\mathcal{Q}}_{\bm \tau}$ parametrizing the regularly stable bundles. The natural map
$\pi^{reg}$ given by the composition
$$\pi^{reg}\,:\,{\mathcal{Q}}_{\bm \tau}^{rs}\,\longrightarrow\,
\mathcal{P}ar_G^{rs}( \bm \tau)\,\longrightarrow\, M_{G,\bm \tau}^{par, rs}$$
is a $L_{\mathcal{C}',G}/Z(G)$ torsor which is \'etale locally trivial. Here,
$L_{\mathcal{C}',G}$ is the loop group associated to a punctured curve.

\subsection{Twisted $\mathcal{D}$-modules via quasi-section of
	Drinfeld-Simpson}

Let $\pi_s\,:\,\mathcal{C}\,\to\, S$ be a versal family of $n$-pointed 
smooth curves of genus $g$. We choose formal coordinates $\bm \xi\,=\,(\xi_1,\,\cdots,\,\xi_n)$ 
along the sections $\bm p$.

Let $\bm \tau\,=\,(\tau_1,\,\cdots,\,\tau_n)$ be as in Section
\ref{sec:uniformization}, and let $\mathcal{Q}_{\tau_i}$ be the affine flag
variety associated to $\tau_i$. The above choice of coordinates gives an
identification of $\mathcal{Q}_{\bm \tau}$ with $\prod_{i=1}^n
L_{G}/\mathcal{P}_{\tau_i}$. By the discussion in
\cite[Secs. 5.2.9-5.2.12]{BD1}, the infinitesimal action, of the central extension
$\widehat{L}_{G}$ of the loop group, on $\mathcal{Q}_{\tau_i}$ gives a map 
$$
\overline{U}(\widehat{\mathfrak{g}}_{\xi_i})^{opp}
\,\longrightarrow \,H^0(\mathcal{Q}_{\tau_i},\,
\mathcal{D}_{\mathcal{Q}_{\tau_i}/S}
(\mathscr{L}_{\vec{\lambda}})).
$$
Here $\mathcal{D}_{\mathcal{Q}_{\tau_i}/S}(\mathscr{\mathscr{L}}_{\lambda_i})$ is the ring 
of relative $\mathscr{L}_{{\lambda}_i}$-twisted differential operators on 
$\mathcal{Q}_{\tau_i}$, and $\overline{U}(\widehat{\mathfrak{g}}_{\xi_i})$ is a suitable 
completion of the universal enveloping algebra of $\widehat{\mathfrak{g}}_{\xi_i}$,
and $\overline{U}(\widehat{\mathfrak{g}}_{\xi_i})^{opp}$ is the opposite
algebra. 
Summing over all the coordinates, we get a map 
\begin{equation}\label{eqn:filter}
\bigoplus_{i=1}^n\left( \overline{U}(\widehat{\mathfrak{g}}_{\xi_i})^{opp}\right)
\,\longrightarrow\, H^0(\mathcal{Q}_{\bm \tau},\,
\mathcal{D}_{\mathcal{Q}_{\bm\tau}/S}(\mathscr{L}_{\vec{\lambda}}))
\end{equation}
which via further restriction gives a map $\bigoplus_{i=1}^n\left( \overline{U}
(\widehat{\mathfrak{g}}_{\xi_i})^{opp}\right)\,\to\,
H^0(\mathcal{Q}^{rs}_{\bm \tau},\, \mathcal{D}_{\mathcal{Q}_{\bm \tau}/S}
(\mathscr{L}_{\vec{\lambda}}))$. Both sides of \eqref{eqn:filter} carry natural
filtrations and the map in \eqref{eqn:filter} is a map of filtered sheaves of algebras.

As in \cite{Laszlo}, we consider a quasi-section of $\pi^{rs}$. The result of
\cite{DS} implies
that the natural \'etale locally trivial torsor $\pi^{rs}\,:\, \mathcal{Q}_{\bm \tau}^{rs}
\,\to\, M_{G,\bm \tau}^{rs}$ has a quasi-section $N^{par,rs}_{G,\bm \tau}
\,\xrightarrow{\,\ r\,\ }\, M_{G,\bm \tau}^{par,rs}$ such that $r$
is an \'etale epimorphism, and there is a map
$\sigma\,:\, N_{G,\tau}^{par,rs}\,\to\, \mathcal{Q}_{\bm \tau}^{rs}$
such that the following diagram commutes
\begin{equation}
\begin{tikzcd}[column sep=large]
& \mathcal{Q}^{rs}_{\bm \tau} \arrow[d, "\pi^{rs}"]\\
N_{G,\bm \tau}^{par,rs} \arrow[r,"r"]\arrow{rd}\arrow[ru, "\sigma"] & M_{G,\bm \tau}^{par,rs} \arrow[d,"\pi_e"]\\
&S
\end{tikzcd}
\end{equation}
Now since the map $r$ is \'etale, we get an isomorphism 
$$H^0(N_{G,\bm \tau}^{par,rs},\, r^*\mathcal{T}_{M_{G,\bm \tau}^{par,rs}/S})\,=\,
H^0(N_{G,\bm \tau}^{par,rs},\, \mathcal{T}_{N_{G,\bm \tau}^{par,rs}/S}).$$

Given any relative differential operator $\mathfrak{D}$ on the line bundle
$\mathscr{L}_{\vec{\lambda}}$, we can pull it back via $\sigma$
(see Section 8.1 and 8.7 in \cite{Laszlo}) to a differential operator on the line bundle
$\sigma^* \mathscr{L}_{\vec{\lambda}}$ which, by an abuse of notation, is again
denote by $\mathscr{L}_{\vec{\lambda}}$. Thus, \eqref{eqn:filter} gives the following map of filtered sheaves of algebras
\begin{equation}\label{zl1}
\hbar_{\mathscr{L}_{\vec{\lambda}}}\,:\, \bigoplus_{i=1}^n
\left( \overline{U}(\widehat{\mathfrak{g}}_{\xi_i})^{opp}\right)\,\longrightarrow\,
H^0(N_{G,\bm \tau}^{par,rs},\,\mathcal{D}_{N_{G,\bm \tau}^{par,rs}/S}(\mathscr{L}_{\vec{\lambda}})).
\end{equation}

The sheaf of Lie algebras $\bigoplus_{i=1}^n\left(
\overline{U}(\widehat{\mathfrak{g}}_{\xi_i})^{opp}\right)$ carries a natural
PBW filtration and we let  $\left(\bigoplus_{i=1}^n\left(
\overline{U}(\widehat{\mathfrak{g}}_{\xi_i})^{opp}\right)\right)^{\leq m}$
be the $m$-th part of the filtration. Then the following diagram is commutative 
\begin{equation}\label{eqn:momentmap}
\begin{tikzcd}[column sep=small]
& H^0(N_{G,\bm \tau}^{par,rs},\mathcal{D}_{N_{G,\bm \tau}^{par,rs}/S}^{\leq m}(\mathscr{L}_{\vec{\lambda}}))\arrow[rd,"symb^{\leq m}"]& \\
\left(\bigoplus_{i=1}^n\left( \overline{U}(\widehat{\mathfrak{g}}_{\xi_i})^{opp}\right)\right)^{\leq m}\arrow[ur,"\hbar_{\mathscr{L}_{\vec{\lambda}}}^{\leq m}"]\arrow[dr,"\hbar^{\leq m}_{\mathcal{O}}"']& & H^0(N_{G,\bm \tau}^{par,rs}, \operatorname{Sym}^m\mathcal{T}_{N_{G,\bm \tau}^{par,rs}/S})\\
& H^0(N_{G,\bm \tau}^{par,rs},\mathcal{D}_{N_{G,\bm \tau}^{par,rs}/S}^{\leq m}(\mathcal{O}_{N_{G,\bm \tau}^{par,rs}}))\arrow[ru,"symb^{\leq m}"']&
\end{tikzcd}
\end{equation}
where $symb^{\leq m}$ denotes the principal $m$-th order symbol map of a differential operator. 

\subsection{Projective heat operator from Sugawara}\label{se8.2}

We now give a local description of the map $\hbar_{\mathscr{L}_{\vec{\lambda}}}$. Let 
$\Pcal$ be a regularly stable parabolic $G$-bundle in the moduli space of
parabolic bundles of 
parabolic weights $\vec{\lambda}$ on a curve $C$ with parabolic structure over $\bm p$. We 
consider it as a point in $N_{G,\bm \tau}^{par,rs}$. The tangent space at $\Pcal$ is given 
by $H^1(C,\, \operatorname{Par}(\Pcal))$, where $\operatorname{Par}(\Pcal)$ is the sheaf of 
Lie algebras given by parabolic endomorphisms of the bundle $\Pcal$.

Let $P_{i}\,\subset\, {G}$ be the parabolic subgroup determined by the weight $\lambda_i$ attached to the point
$p_i\,\in\, \bm p$, and let $\mathfrak{p}_i$ be the corresponding Lie algebra. We denote by $\mathfrak{p}_i^{-}$
the opposite parabolic and by $\mathfrak{n}_{i}^{-}$ the nilpotent radical of
$\mathfrak{p}_i^{-}$. We have a short exact sequence of sheaves 
	\begin{equation}\label{zl}
	0 \,\longrightarrow\, \operatorname{Par} (\Pcal)\,\longrightarrow\,
\operatorname{Par}(\Pcal)\left(\sum_{i=1}^nm_ip_i\right)\,\longrightarrow\,
\bigoplus_{i=1}^n\left(\mathfrak{n}_{i}^{-}\oplus
\bigoplus_{j=1}^{m_i}  \mathfrak{g}\otimes \xi_i^j \right)\,\longrightarrow\, 0,
\end{equation}
where $m_1,\,\cdots,\,m_n$ are nonnegative integers. Taking the long exact sequence
of cohomologies associated to \eqref{zl}, we get a homomorphism
	\begin{equation}
	\bigoplus_{i=1}^n\left(\mathfrak{n}_{i}^{-}\oplus \mathfrak{g}\otimes\mathbb{C}[\xi_i^{-1}]\xi_i^{-1}\right)
\,\longrightarrow\, H^1(C,\, \operatorname{Par}(\Pcal)).
\end{equation}
Combining this with the natural projection $\mathfrak{g}\otimes \mathbb{C}((\xi_i))\,
\to\,
\mathfrak{n}_i^{-}\oplus \mathfrak{g}\otimes \mathbb{C}[\xi_i^{-1}]\xi_i^{-1}$ for each
$1\,\leq\, i\,\leq\, n$, we get a homomorphism 
\begin{equation}\label{eqn:projection1}
\rho_i\,:\, \mathfrak{g}\otimes \mathbb{C}((\xi_i))\,\longrightarrow\, \mathfrak{n}_i^{-}\oplus
\mathfrak{g}\otimes \mathbb{C}[\xi_i^{-1}]\xi_i^{-1}\,\longrightarrow\, H^1(C,\,\operatorname{Par}(\Pcal)).
\end{equation}
The composition of maps $\rho_i$ in \eqref{eqn:projection1} is the local description of
$\hbar_{\mathcal{O}}$ (defined in \eqref{zl1})
\begin{equation}\label{eqn:obvious1}
\left(\overline{U}(\widehat{\mathfrak{g}}_{p_i})^{opp}\right)^{\oplus n}\,\longrightarrow\,
H^0(N_{G, \bm \tau }^{par,rs},\, \mathcal{D}_{N_{G, \bm \tau }^{par,rs}/S}(\mathcal{O})).
\end{equation}

The operator $\mathscr{D}(\vec{\ell})$ defined in \eqref{dl} gives a relative second order
differential operator $\mathfrak{D}$ on
$N_{G,\bm \tau}^{par,rs}$ which acts on the $i$-th factor by $(T[\underline{l}_i])$ (see \eqref{tl} and \eqref{eqn:projection1}). Thus we have the following diagram
\begin{equation}
\begin{tikzcd}
\bigoplus_{i=1}^n \mathcal{O}_S((\xi_i))\frac{d}{d\xi_i}\arrow[r,"\mathfrak{D}"] \arrow[d,"\theta"]& H^0(N_{G, \bm \tau }^{par,rs},\,
\mathcal{D}^{\leq 2}_{N_{G, \bm \tau}^{par,rs}/S}(\mathscr{L}_{\vec{\lambda}}))\\
H^0(S,\,T_S).
\end{tikzcd}
\end{equation}
We can realize $\mathfrak{D}$ as a projective heat operator by taking a lift of a vector 
field on $S$ to an element of $\bigoplus_{i=1}^n \mathcal{O}_S((\xi_i))\frac{d}{d\xi_i}$. 
Now as described in the previous section, the difference between two lifts can be understood as 
a $\mathcal{O}_S$-module homomorphism $a_{\omega}$. Thus the map $\mathfrak{D}$ descends to 
a projective heat operator, and we will also denote the descended
operator by ${\mathfrak{D}}$. In the rest 
of this section we show that the symbol of $\mathfrak{D}$ is the Hitchin symbol 
$\widetilde{\rho}_{sym}$, which will complete the proof of Theorem \ref{thm:maintheorem}.

\subsection{The parabolic duality pairing and the Hitchin symbol}

Recall that the Cartan-Killing form induces a nondegenerate bilinear form between the sheaves
$$
\kappa_{\mathfrak{g}}\,:\, \operatorname{Spar}(\Pcal)(D)\otimes \operatorname{Par}(\Pcal)
\,\longrightarrow\, \mathcal{O}_C.
$$
Let $D_{p_i}$ be a formal disc around each marked point $p_i$ in $C$, and let $C^*\,=\,
C\backslash \{p_1,\,\cdots,\,p_n\}$ be the complement. Consider the following
open covering: $$C\,=\,C^*\cup \left( \sqcup_{i=1}^n D_{p_i}\right).$$ 
A section of $\operatorname{Spar}(\Pcal)$ restricted to $D_{p_i}$ consists of an element of
$\mathfrak{g}\otimes \mathbb{C}[[\xi_i]]$ whose image under the natural evaluation map
$$\operatorname{ev}_{p_i}\,:\,
\mathfrak{g}\otimes \mathbb{C}[[\xi_i]]\,\longrightarrow\, \mathfrak{g}$$
is contained in the nilradical $\mathfrak{n}_i$ of the parabolic subalgebra
$\mathfrak{p}_i$. Similarly, $\operatorname{Par}(\Pcal)$ consists of sections whose
restriction to any formal disc $D_{p_i}$ has the
property that the image of the evaluation map is in $\mathfrak{p}_i$.  

Let $\{\overline{P}_{i}\}$ be a \v{C}ech cocycle representative in $\prod_{i=1}^n 
\left(\operatorname{Par}(\Pcal) (D_{p_i}^*)\right)$ of a cohomology class of $H^1(C,\, 
\operatorname{Par}({\Pcal}))$ with respect to the covering $C\,=\,C^*\cup \left( \sqcup 
D_{p_i}\right)$.
Here we have $\overline{P}_i \,\in\, \mathfrak{g}\otimes
\mathbb{C}((\xi_i))$, under a trivialization of $\Pcal$ restricted to $D_{p_i}$. Similarly we let
$\{\phi_i d\xi_i\}\,\in\, \mathfrak{g}\otimes \mathbb{C}((\xi_i))d\xi_i$ denote the
restriction of an element of
$H^0(C,\, \operatorname{Spar}(\Pcal)\otimes K_C(D))$ to $\sqcup D_{p_i}^*$.

The natural pairing in \eqref{eqn:naturalcech} takes the form 
\begin{eqnarray}
\label{eqn:naturalcech2} 
H^0(C,\,\operatorname{Spar}(\Pcal)\otimes K_C(D))\otimes H^1(C,\,\operatorname{Par}(\Pcal))
\,\lra\,  \mathbb{C}\\
\{\phi_id_{\xi_i}\} \times  \{\overline{P}_i\}\,\longmapsto\, \sum_{i=1}^n \operatorname{Res}_{\xi_i=0}
\kappa_{\mathfrak{g}}(\phi_i,\overline{P}_i)d\xi_i.
\end{eqnarray}
Now consider a \v{C}ech representative $\vec{\ell}\,=\,\{\underline{l}_i\}
\,\in\, \oplus_{i=1}^n \mathbb{C}((\xi_i))\frac{d}{d\xi_i}$ of a cohomology class in
$H^1(C,\, T_C(-D))$.  Let $\phi$ be a global section of the sheaf
$\operatorname{Spar}(\Pcal)\otimes K_C(D)$. For each $i$, we have
$$\underline{l}_i\,=\,\sum_{m=-m_i}^{\infty} l_{i,m}\xi_i^{m+1}\frac{d}{d\xi_i},$$ and
$\phi$ restricted to $D_{p_i}^*$ is of the form 
$$\phi_id\xi_i\,=\,\sum_{m\in \mathbb{Z}} X_{i,m} \xi_i^{-m-1}d\xi_i \in \mathfrak{g}\otimes \mathbb{C}((\xi_i))d\xi_i.$$

Since the diagram in \eqref{eqn:momentmap} commutes, we can evaluate the symbol
of $\mathfrak{D}$ by computing the following: 
\begin{eqnarray*}
\langle \phi \otimes \phi ,\, \sum_{i=1}^n T[\underline{l}_i]\rangle &=&\sum_{i=1}^n \langle \phi_i\otimes \phi_id\xi_i^{2},\, \sum_{m=-m_i}^{\infty} l_{i,m}L_m\rangle \\
&=& \sum_{i=1}^n\sum_{m=-m_i}^{\infty} l_{i,m} \langle \phi_i \otimes \phi_id\xi_i^{2},\,  L_m\rangle \\
&=&\sum_{i=1}^n\sum_{m=-m_i}^{\infty}l_{i,m} \langle \phi_i \otimes \phi_id\xi_i^2,\, \frac{1}{2(\ell+h^{\vee}(\mathfrak{g}))}\sum_{k\in \mathbb{Z}} \sum_{a=1}^{\dim \mathfrak{g}} \norder {J^a(k)J^a(m-k)}\rangle. \\
\end{eqnarray*}

Now if $\underline{l}_i\,=\,\xi_i^{n_i+1}\frac{d}{d\xi_i}$, we get that
\begin{eqnarray*}
\langle \phi\otimes \phi ,\, \sum_{i=1}^n T[\xi_i^{n_i+1}\frac{d}{d\xi_i}]\rangle &=& \sum_{i=1}^n \langle \phi_i \otimes \phi_id\xi_i^2,\, L_{n_i}\rangle \\
&=& \frac{1}{2(\ell+h^{\vee}(\mathfrak{g}))}\sum_{i=1}^n\sum_{k\in \mathbb{Z}}
\sum_{a=1}^{\dim \mathfrak{g}}\langle \phi_i\otimes \phi_id\xi_i^2,\, \norder {J^a(k)J^a(n_i-k)}\rangle. 
\end{eqnarray*}

If $n_i\,\neq\, 0$, then we get that
\begin{align*}
\langle \phi_i\otimes \phi_id\xi_i^2,\, L_{n_i}\rangle
&=\,\frac{1}{2(\ell+h^{\vee}(\mathfrak{g}))}\sum_{k\in \mathbb{Z}}\sum_{a=1}^{\dim \mathfrak{g}}\langle \phi_i\otimes \phi_i d\xi_i^2,\, J^a(k)J^a(n_i-k)\rangle\\
&=\frac{1}{2(\ell+h^{\vee}(\mathfrak{g}))}\sum_{k\in \mathbb{Z}}\sum_{a=1}^{\dim \mathfrak{g}}\operatorname{Res}_{\xi_i=0}\langle {\phi}_i,\, J^a(k) \rangle d\xi_i.\operatorname{Res}_{\xi_i=0}\langle {\phi}_i,\, J^a(n_i-k) \rangle d\xi_i\\
&=\frac{1}{2(\ell+h^{\vee}(\mathfrak{g}))}\sum_{k\in \mathbb{Z}}\sum_{a=1}^{\dim \mathfrak{g}}\left(\sum_{m\in \mathbb{Z}}\operatorname{Res}_{\xi_i=0}\kappa_{\mathfrak{g}}(X_{i,m},\,J^a)\xi^{k-m-1}d\xi_i\right)\\
&\qquad\qquad \times \left(\sum_{m\in \mathbb{Z}}\operatorname{Res}_{\xi_i=0}\kappa_{\mathfrak{g}}(X_{i,m},\,J^a)\xi^{n_i-k-m-1}d\xi_i\right)\\
&=\frac{1}{2(\ell+h^{\vee}(\mathfrak{g}))}\sum_{k\in \mathbb{Z}}\sum_{a=1}^{\dim \mathfrak{g}}\kappa_{\mathfrak{g}}(X_{i,k},\,J^a)\kappa_{\mathfrak{g}}(X_{i,n_i-k},\,J^a)\\
&=\frac{1}{2(\ell+h^{\vee}(\mathfrak{g}))}\sum_{k\in \mathbb{Z}} \kappa_{\mathfrak{g}}(X_{i,k},\, X_{i,n_i-k}).
\end{align*}

The zero-th Virasoro operator $L_0$ can be rewritten without normal ordering as follows: 
$$L_0\, =\, \frac{1}{2(\ell+h^{\vee}(\mathfrak{g}))}\sum_{a=1}^{\dim \mathfrak{g}}J^aJ^a+
\frac{1}{(\ell+h^{\vee}(\mathfrak{g}))}\sum_{k=1}^{\infty}J^a(-k)J^a(k).$$

Thus we get the following:
\begin{align*}
&\langle \phi_i\otimes \phi_i d\xi_i^2,\, L_0\rangle \\
&=\frac{1}{2(\ell+h^{\vee}(\mathfrak{g}))}\sum_{a=1}^{\dim \mathfrak{g}}\langle \phi_i\otimes \phi_i d\xi_i^2,\, J^aJ^a\rangle+\frac{1}{(\ell+h^{\vee}(\mathfrak{g}))}\sum_{k=1}^{\infty}
\langle \phi_i\phi_id\xi_i^2,\, J^a(-k)J^a(k)\rangle\\
&= \frac{1}{2(\ell+h^{\vee}(\mathfrak{g}))}\sum_{a=1}^{\dim \mathfrak{g}}\left(\sum_{m\in \mathbb{Z}}\operatorname{Res}_{\xi_i=0}\kappa_{\mathfrak{g}}(X_{i,m},\,J^a)\xi_i^{-m-1}d\xi_i\right)^{2}
\\&\quad  +\frac{1}{(\ell+h^{\vee}(\mathfrak{g}))}\sum_{k=1}^{\infty}\sum_{a=1}^{\dim \mathfrak{g}}\Bigg(\left(\sum_{m\in \mathbb{Z}}\operatorname{Res}_{\xi_i=0}\kappa_{\mathfrak{g}}(X_{i,m},\,J^a)\xi_i^{-k-m-1}d\xi_i\right)\\
&\quad  \quad \times\left(\sum_{m\in \mathbb{Z}}\operatorname{Res}_{\xi_i=0}\kappa_{\mathfrak{g}}(X_{i,m},\,J^a)\xi_i^{k-m-1}d\xi_i\right)\Bigg)\\
& 
=\frac{1}{2(\ell+h^{\vee}(\mathfrak{g}))}\sum_{a=1}^{\dim\mathfrak{g}}\kappa_{\mathfrak{g}}
(X_{i,0},\,J^a)\kappa_{\mathfrak{g}}(X_{i,0},\,J^a).\\& 
\quad +\frac{1}{(\ell+h^{\vee}(\mathfrak{g}))}\sum_{k=1}^{\infty}\sum_{a=1}^{\dim 
\mathfrak{g}}\kappa_{\mathfrak{g}}(X_{i,-k},\,J^a)\kappa_{\mathfrak{g}}(X_{i,k},\,J^a)\\
&=\frac{1}{2(\ell+h^{\vee}(\mathfrak{g}))}\kappa_{\mathfrak{g}}(X_{i,0},\,X_{i,0})+\frac{1}{2(\ell+h^{\vee}(\mathfrak{g}))}\sum_{k\in 
\mathbb{Z}\backslash \{0\}}\kappa_{\mathfrak{g}}(X_{i,-k},\,X_{i,k})\\
&=\frac{1}{2(\ell+h^{\vee}(\mathfrak{g}))}\sum_{k\in \mathbb{Z}} \kappa_{\mathfrak{g}}(X_{i,k},\,X_{i,-k}).
\end{align*}

We summarize the above calculations in the following proposition.

\begin{proposition}
For any $1\,\leq\, i \,\neq\, n$, and for any $m_i\,\in\, \mathbb{Z}$,
$$\langle \phi_i\otimes \phi_i d\xi_i^2,\, L_{m_i}\rangle\,=\,
\frac{1}{2(\ell+h^{\vee}(\mathfrak{g}))}\sum_{k\in \mathbb{Z}}
\kappa_{\mathfrak{g}}(X_{i,k},\,X_{i,m_i-k}).$$
\end{proposition}

\subsection{Proof of the Main theorem (Theorem \ref{thm:maintheorem})}
\label{sec:mainproof}

Recall that the product
$$
H^1(C,\, T_C(-D))\otimes H^0(C,\, \operatorname{SPar}(\Pcal)\otimes K_{C}(D))\,\longrightarrow\,
H^1(C,\, \operatorname{Par}\Pcal)$$ induces a homomorphism 
\begin{equation}
\widetilde{\rho}_{sym}\,:\, R^1\pi_{n*}\mathcal{T}_{\mathfrak{X}_{G}^{par}/
M_{G, \bm \tau }^{par,rs}}(-D)\,\longrightarrow\, \pi_{n*}\operatorname{Sym}^2\mathcal{T}_{M_{G,\bm \tau}^{par,rs}/S}.
\end{equation}
Consider the \v{C}ech cover of $C$ given by $C\, =\ C^*\cup \left(\sqcup_{i=1}^n D_{p_i}\right)$.
In particular, given any \v{C}ech cohomology class $\{\xi_i^{n_i+1}\frac{d}{d\xi_i}\}$ in
$H^1(C, \,T_C(-D))$, using Serre duality and the identification of
$\operatorname{SPar}(\Pcal)(D)$ with $\operatorname{Par}(\Pcal)^{\vee}$, we get a
symmetric bilinear form on $H^0(C,\, \operatorname{SPar}(\Pcal)\otimes K_C(D))$.

As in the previous section, consider a section $\phi \in H^0(C,\,
\operatorname{SPar}(\Pcal)\otimes K_C(D)$. For each $i$, the section $\phi$ restricted to $D_{p_i}^*$ is of the form 
$$\phi_id\xi_i\,=\, \sum_{m\in \mathbb{Z}} X_{i,m} \xi_i^{-m-1}d\xi_i \in \mathfrak{g}\otimes \mathbb{C}((\xi_i))d\xi_i.$$
 Thus evaluating a cocycle class $\{\xi_i^{n_i+1}\frac{d}{d\xi_i}\}$ against a section $\phi$ written in the \v{C}ech cover as $\{\phi_id\xi_i\}$, we get that 
 \begin{eqnarray*}
 \{\xi_i^{n_i+1}\frac{d}{d\xi_i}\}(\phi)&:=&\operatorname{Res}_{\xi_i=0}\kappa_{\mathfrak{g}}(\phi_id\xi_i\otimes \langle \xi_i^{n_i+1}\frac{d}{d\xi_i},\,\phi_id\xi_i\rangle )\\
 &=&\operatorname{Res}_{\xi_i=0} \kappa_{\mathfrak{g}}(\phi_id\xi_i,\, \sum_{m \in \mathbb{Z}}X_{i,m}\xi_i^{n_i-m})\\
 &=&\operatorname{Res}_{\xi_i=0} \kappa_{\mathfrak{g}}(\sum_{k\in \mathbb{Z}}X_{i,k}\xi_i^{-k-1},
\, \sum_{m \in \mathbb{Z}}X_{i,m}\xi_i^{n_i-m})d\xi_i\\
 &=&\sum_{k\in \mathbb{Z}}\kappa_{\mathfrak{g}}(X_{i,k},\,X_{i,n_i-k}).
 \end{eqnarray*}
We summarize the discussion in this subsection in the following proposition which completes the proof the Theorem \ref{thm:maintheorem}. 

\begin{proposition}\label{prop:equality}
Let $a$ be any rational number and $\phi$ be a faithful representation with Dynkin index $m_{\phi}$. Then the  symbol of the projective heat operator
${\mathfrak{D}}$ acting on $\mathscr{L}_{\vec{\lambda}}^{\otimes a.{m_{\phi}}}\cong\operatorname{Det}^{\otimes a}_{par,\phi}(\bm \tau)$ constructed from
the Sugawara tensor and uniformization (see Section \ref{se8.2})
coincides with $$\frac{1}{2(a\cdot m_{\phi}\cdot \ell+
h^{\vee}(\mathfrak{g}))}\widetilde{\rho}_{sym}.$$ Hence, the projective heat operator $\mathfrak{D}$
and the projective heat operator constructed in \cite{BMW1} via \eqref{eqn:fundamental} coincide.
\end{proposition}

\begin{remark}\label{rem:nonample}
If a square-root of $K_{M^{par,s}_{\SL_r, \bm{\alpha}}/S}$ exists, then it follows that the push-forward of the line bundle
$\operatorname{Det}_{par}(\bm \alpha)\otimes K^{\frac{1}{2}}_{M^{par,s}_{\SL_r,\bm{\alpha}}/S}$ produces conformal blocks of 
level $\ell-h^{\vee}(\mathfrak{sl}(r))$, where $h^{\vee}(\mathfrak{g})$ is the dual Coxeter number of a Lie algebra $\mathfrak{g}$. From the calculations in this section, and the 
fact that the tangent and cotangent spaces of the moduli space $M^{par,s}_{\SL_r, \bm{\alpha}}$ are only dependent on the 
flag type of $\bm{\alpha}$, it follows that the symbol ${\rho_{sym}/\ell}$ equals the symbol of the 
differential operator that induces the TUY connection. However, it should be
mentioned that even if 
$K_{M^{par,s}_{\SL_r, \bm{\alpha}}/S}$ has a square-root, the pushforward $\pi_*\left( 
\operatorname{Det}_{par}(\bm \alpha)\otimes K^{\frac{1}{2}}_{M^{par,s}_{\SL_r, \bm{\alpha}}/S}\right)$ may not have any sections. 
\end{remark}
\section{Geometrization of the KZ equation on invariants}\label{sec:geometrizationofKZ}

In this section, we show a geometric construction of the Knizhnik-Zamolodchikov connection 
(KZ). This question was suggested to us by Professor P. Belkale. Let us first recall the 
classical construction of the KZ connection \cite{KZ}.

\subsection{KZ connection} Let $\mathfrak{g}$ be a fixed semisimple Lie algebra, and
let $\vec{\lambda}\,=\,(\lambda_1,\,\cdots,\,\lambda_n)$ be an $n$-tuple of highest
weights. Consider the vector space of invariants of tensor product of representations 
$$A_{\vec{\lambda}}(\mathfrak{g})\,:=\, \operatorname{Hom}_{\mathfrak{g}}(V_{\lambda_1}\otimes \dots \otimes V_{\lambda_n}, \mathbb{C}).$$
The space of invariants sits inside the zero weight space $(V_{\lambda_1}\otimes \dots \otimes V_{\lambda_n})^*_0$ of the dual of the tensor product of representations. 

Let $X_{n}\,=\,\{\bm z\,=\,(z_1,\,\cdots,\, z_n)\,\in\, \mathbb{C}^n\,\,\mid\,\,
z_i\,\neq\, z_j \}$, 
and consider the trivial vector bundle $\mathbb{A}_{\vec{\lambda}}$ on the 
configuration space of points $X_{n}$ whose fiber is $A_{\vec{\lambda}}(\mathfrak{g})$.
It is well known \cite{FSV1, FSV2, TUY:89} 
that the space of conformal blocks
$\mathcal{V}^{\dagger}_{\vec{\lambda}}(C,\mathfrak{g},\ell, \bm z)$ on 
$\mathbb{P}^1\,=\,\mathbb{C}\sqcup \{\infty\}$ with $n$ marked points
$(z_1,\,\cdots,\, z_n)$ for $\mathfrak{g}$ at level $\ell$ and weights
$\vec{\lambda}$   injects into $A_{\vec{\lambda}}(\mathfrak{g})$:
\begin{equation}\label{eqn:inject}
\iota\,:\, \mathcal{V}^{\dagger}_{\vec{\lambda}}(\mathbb{P}^1,\mathfrak{g},\ell, \bm z)
\,\hookrightarrow\, A_{\vec{\lambda}}(\mathfrak{g}).
\end{equation}
This map is actually an isomorphism for $\ell\,\gg\, 0$. Specific bounds for $\ell$
are given in  Belkale-Gibney-Mukhopadhyay (\cite{BGM2, BGM1}).

As in Section \ref{sec:connection}, consider an orthonormal basis
$J^1,\,\cdots,\, J^{\dim \mathfrak{g}}$ of the Lie algebra $\mathfrak{g}$ for
the normalized Cartan-Killing form. Define the Casimir operator $\Omega \,=\,
\sum_{a=1}^{\dim \mathfrak{g}} J^aJ^a$. For pairs of integers $1\,\leq\, i\,\neq\,
j\,\leq\, n$, and vectors $v_1\otimes \cdots \otimes v_n\,\in\,
V_{\lambda_1}\otimes \cdots \otimes V_{\lambda_n}$, let 
$$\Omega_{i,j}(v_1\otimes\cdots \otimes v_n)\,:=\,
\sum_{a=1}^{\dim \mathfrak{g}} v_1\otimes\cdots \otimes J^av_i\otimes \cdots \otimes J^av_j\otimes \cdots \otimes v_n .$$
For any complex number $\kappa\,\neq\, 0$, the formula 
\begin{equation}
\label{eqn:KZeqn}
\left(\nabla^{(\kappa)}_{\frac{\partial}{\partial z_i}}(f\otimes
\langle \Psi| \right)(|\Phi\rangle)\,:=\,
\frac{\partial f}{\partial z_i} \langle \Psi | \Phi\rangle-\frac{f}{\kappa} \sum_{j\neq i}
\frac{\langle \Psi|\Omega_{i,j}(\phi)\rangle}{z_i-z_j},
\end{equation}
defines a flat connection on $(V_{\lambda_1}\otimes \cdots \otimes V_{\lambda_n})^*_0\otimes
\mathcal{O}_{X_n}$ 
over $X_n$ that preserves the subbundle $\mathbb{A}_{\vec{\lambda}}$.  
Hence, its monodromy  gives a representation of the pure braid group $\pi_1(X_n,\,\bm z)$. 

In this discussion, we restrict ourselves to the case 
where $\kappa\,=\,\ell+ h^{\vee}(\mathfrak{g})$, and 
$\kappa_{\mathfrak{g}}(\lambda_i,\,\theta_{\mathfrak{g}})\,<\,1$ for all $i$.
In this case
it is known that  the connection $\nabla^{(\ell+h^{\vee}(\frak{g}))}$ preserves the bundle $\mathcal{V}_{\vec{\lambda}}^{\dagger}(\mathfrak{g},\ell)$
of conformal blocks and it is equal to  the TUY/WZW
connection \cite{FSV2, FSV1, TUY:89}.

\subsection{Invariants as global sections}

As in Section
\ref{sec:parabolicmoduli} consider the moduli stack of quasi-parabolic
bundles $\mathcal{P}ar_{G}(\mathbb{P}^1, \bm z, \bm \tau)$ of local type
$\bm \tau$ on $\mathbb{P}^1$, where $\bm \tau$ and $\vec{\lambda}$ are
related  by the usual exponential map as before. Consider the open substack $\mathcal{P}ar^{c}_G(\mathbb{P}^1, \bm z, \tau)$ of $\mathcal{P}ar_G(\mathbb{P}^1, \bm z, \bm \tau)$ parametrizing quasi-parabolic bundle on $\mathbb{P}^1$ whose underlying bundle is trivial. By construction, we have an isomorphism of $\mathcal{P}ar_G^c(\mathbb{P}^1, \bm z, \bm \tau)$ with the quotient stack 
\begin{equation}
[\left(G/P_1\times \cdots \times G/P_n\right)/G],
\end{equation}
where $P_1,\,\cdots,\, P_n$ are the parabolics determined by $\tau_1,\,\cdots,\, \tau_n$
and $G$ acts diagonally on the product of partial flag varieties.

Let $\mathscr{L}_{\vec{\lambda}}$ be the Borel-Weil-Bott line bundle on
$\mathcal{P}ar_G(\mathbb{P}^1, \bm z, \bm \tau)$, and consider the restriction of $\mathscr{L}_{\vec{\lambda}}$ to $\mathcal{P}ar_G^{c}(\mathbb{P}^1, \bm z, \bm \tau)$. 
We get a natural map 
\begin{equation}\label{eqn:res1}
H^0(\mathcal{P}ar_G(\mathbb{P}^1, \bm z, \bm \tau),\,\mathscr{L}_{\vec{\lambda}})
\,\lra\, H^0(\mathcal{P}ar_G^{c}(\mathbb{P}^1, \bm z, \bm \tau),\, \mathscr{L}_{\vec{\lambda}}).
 \end{equation}
 Now the restriction of $\mathscr{L}_{\vec{\lambda}}$ to $[\left(G/P_1\times \cdots \times G/P_n\right)/G]$ is  $L_{\lambda_1}\boxtimes \dots \boxtimes L_{\lambda_n}$,
 where the  $L_{\lambda_i}$ are the natural homogeneous line bundles on
 $G/P_{i}$ determined by the weights $\lambda_i$. Moreover, by the
 Borel-Weil theorem, we have $H^0(G/P_i,\, L_{\lambda_i})\,=\,V_{\lambda_i}^*$.
 Thus, from the restriction we get the natural commutative diagram 
 \begin{equation}
 \begin{tikzcd}
 H^0(\mathcal{P}ar_G(\mathbb{P}^1, \bm z, \bm \tau), \mathscr{L}_{\vec{\lambda}})\arrow[d,"\cong"]\arrow[r,"\operatorname{res}"]& H^0(\mathcal{P}ar_G^{c}(\mathbb{P}^1, \bm z, \bm \tau), \mathscr{L}_{\vec{\lambda}})\arrow[r, equal]\arrow[d,"\cong"]& \left(\bigotimes_{i=1}^nH^0(G/P_i, L_{\lambda_i})\right)^{\mathfrak{g}}\arrow[d,"\cong"]\\
 \mathcal{V}_{\vec{\lambda}}^{\dagger}(\mathbb{P}^1, \mathfrak{g}, \ell,\bm z) \arrow[r, hook,"\iota"]& \operatorname{Hom}_{\mathfrak{g}}(V_{\lambda_1}\otimes \cdots \otimes V_{\lambda_n}, \mathbb{C})\arrow[r,equal] & \big(V^*_{\lambda_1}\otimes \cdots \otimes V^*_{\lambda_n}\big)^{\mathfrak{g}}.
 \end{tikzcd}
 \end{equation}
Here the left vertical isomorphism is due to Laszlo-Sorger \cite{LaszloSorger:97};
the diagram was used in \cite{BGM2}.  Now it follows that the complement of $\mathcal{P}ar_G^c(\mathbb{P}^1, \bm z, \bm \tau)$ in $\mathcal{P}ar_G(\mathbb{P}^1, \bm z, \bm \tau)$  is just the ordinary theta divisor. 

\subsection{Differential operators}

Recall the notion of a {\em good
 stack} from Beilinson-Drinfeld \cite{BD1}: 
 an equidimensional algebraic stack $\mathcal{Y}$ over complex numbers is {\em good} if the
dimension of $\mathcal{Y}$ is half the dimension of the cotangent stack
$\mathcal{T}^{\vee}_\mathcal{Y}$. Let $\mathcal{Y}_{sm}$ be the smooth topology of
$\mathcal{Y}$. For any object $S\,\in\, \mathcal{Y}_{sm}$ and a smooth $1$-morphism
$\pi_S \,\in\, \mathcal{Y}$, we have the exact sequence 
$$
\mathcal{T}_{S/\mathcal{Y}} \,\lra\, \mathcal{T}_{S}\,\longrightarrow\,
\pi_S^*\mathcal{T}_Y \,\lra\, 0.
$$
Consider the sheaf of differential operators $\mathcal{D}_S$ on $S$ and the left ideal 
$I\,=\,\mathcal{D}_{S}\mathcal{T}_{S/\mathcal{Y}}\,\subset\, \mathcal{D}_S$. Set 
$\mathcal{D}_{\mathcal{Y}}(S)\,:=\,\mathcal{D}_S/I,$. This $\mathcal{D}_{\mathcal{Y}}$ is an 
$\mathcal{O}_{\mathcal{Y}}$ module along with a natural filtration such that 
$\operatorname{Sym} \mathcal{T}_{\mathcal{Y}}\,\cong\, \operatorname{gr} 
\mathcal{D}_{\mathcal{Y}}$. The above also works for differential operators twisted by a 
line bundle.
 
Since the nilpotent cone of the moduli space of parabolic Higgs bundles is isotropic of exactly 
half  the dimension \cite{BaragliaKamgarpourVarma:19, Faltings:93,LogaresMartens:10}, 
 it follows that the stack $\mathcal{P}ar_{G}(\mathbb{P}^1, \bm z, \bm
 \tau)$ is good.  Moreover, since both $\mathcal{P}ar^{c}_{G}(\mathbb{P}^1,
 \bm z, \bm \tau)$ and $\mathcal{P}ar^{rs}_{G}(\mathbb{P}^1, \bm z, \bm
 \tau)$ are quotients of a smooth scheme by a reductive group, they are also good.
 
Now we know that the line bundle $\mathscr{L}_{\vec{\lambda}}$ descends to a line bundle
on $\mathcal{P}ar_{G}(\mathbb{P}^1, \bm z, \bm \tau)$. 
The construction of the projective heat operator (cf.\ Definition \ref{def:addingadj}) with 
symbol \eqref{eqn:fundamental} gives a second order differential operator $\mathbb{D}$ on 
$\Lscr_{\vec{\lambda}}$ over the moduli stack $\mathcal{P}ar^{rs}_{G}(\mathbb{P}^1, \bm z, 
\bm \tau)$. Since the sheaf $\mathcal{D}^{\leq 2}(\mathscr{L}_{\vec{\lambda}})$ is 
coherent and $\mathcal{P}ar^{rs}_{G}(\mathbb{P}^1, \bm z, \bm \tau)$ has complement of 
dimension at least two (provided $\bm \tau$ satisfies the conditions in the statement of 
Theorem \ref{thm:uniformverlinde}) applying Hartogs theorem, we get a differential operator on 
$\mathscr{L}_{\vec{\lambda}}$ over the entire stack $\mathcal{P}ar_G(\mathbb{P}^1,\bm z, 
\tau)$, which we will still denote by $\mathbb{D}$.
 
 Recall that via the uniformization theorem and the Sugawara
 construction, we have a degree two differential operator $\mathscr{D}$ on
 $\mathscr{L}_{\vec{\lambda}}$, which by Theorem \ref{thm:maintheorem}
 agrees with $\mathbb{D}$. Since the Sugawara construction restricted to
 the open substack $\mathcal{P}ar_G^{c}(\mathbb{P}^1, \bm z, \bm \tau)$
 induces the KZ connection, we have the following corollary obtained by
 restricting $\mathbb{D}$ to $\mathcal{P}ar_G^{c}(\mathbb{P}^1,\bm z, \bm
 \tau)$. 

\begin{corollary}\label{cor:geoKZ}
Let $\pi^c\,:\, \mathcal{P}ar_G^{c}(\bm \tau)\,\to\, X_n$ be the relative open substack of quasi-parabolic bundles whose underlying bundle is trivial. Then
the heat operator $\mathbb{D}$ induces a flat connection on the vector bundle $\pi^{c}_*\mathscr{L}_{\vec{\lambda}}$ over $X_n$ whose fiber at a point $\bm z$ is $H^0(\mathcal{P}ar^{c}_{G}(C,\bm z, \bm \tau), \mathscr{L}_{\vec{\lambda}})$. 
Moreover, the natural identification of $\pi_{*}^c\mathscr{L}_{\vec{\lambda}}$ with $\mathbb{A}_{\vec{\lambda}}$ is flat for the geometric connection on $\pi_{*}^c\mathscr{L}_{\vec{\lambda}}$ and the Knizhnik-Zamolodchikov connection on $\mathbb{A}_{\vec{\lambda}}$.
 \end{corollary}
\appendix
\section{Moduli spaces of parabolic bundles} \label{sec:moduli}

In this section, we briefly recall the basic notion of parabolic bundles and the natural 
line bundles on their moduli spaces.

Let $C,\bm p$ be as in Section \ref{se3.1}. Let $E$ be a vector bundles on $C$. A {\em 
	quasi-parabolic structure} on $E$ at a point $p\,\in\, \bm{p}$ is a strictly decreasing 
flag
$$E_{p}\, =\, F^1E_{p}\,\supset\, F^2E_{p}\,\cdots\, \supset\,\cdots\, \supset\, F^{k_{p}}E_{p}\,\supset\,
F^{k_{p}+1}E_x\,=\,0.$$
of linear subspaces 
in $E_p$. The above integer $k_{p}$ is the length of the flag at $p$, and the tuple
$$(r_1(F^{\bullet}E_{p}),\,\cdots,\, r_{k_{p}}(F^{\bullet}E_{p}))$$
records the jumps in the dimension of the subspaces and is defined by 
$$r_j(F^{\bullet}E_{p})\,:=\,\dim F^j E_p-\dim F^{j+1}E_p.$$
A {\em parabolic structure} on $E$ at $p$ is a quasi-parabolic structure
as above together with a sequence of rational numbers
$$0\,\leq\, \alpha_1 \,<\,\alpha_2\, <\, \cdots\, <\, \alpha_{k_p} < 1$$
known as the {\em weights}. A parabolic bundle $(E,\, {\bm \alpha},\, \bm r)$ on $C$
with parabolic divisor $\bm p$ is a vector bundle $E$ on $C$ along
with parabolic structure over the points in $\bm p$. Using the weights $\bm \alpha$, the {\em parabolic degree} of $E$
is defined to be
$$
{\rm pdeg}(E)\,:=\,\deg (E)+ \sum_{i=1}^n 
\sum_{j=1}^{k_{p_i}}r_j(F^{\bullet}(E_{p_i}))\alpha_j(F^{\bullet}(E_{p_i})).$$ 
Stable and semistable parabolic bundles are defined using the 
parabolic degree (see \cite{MehtaSeshadri}). Mehta and Seshadri constructed the moduli space $M_{\bm \alpha}^{par}$ of 
semistable parabolic bundles \cite{MehtaSeshadri}. 

We now discuss some natural ample line bundles on $M_{\bm \alpha}^{par}$, following \cite{BiswasRaghavendra}. 
Let ${\bm \alpha}$ be a fixed set of weights for fixed flag type $\bm r$, and let $(E,\, \bm {\alpha},\, \bm{r})$ be a parabolic 
bundle on $(C,\, \bm{p})$. 
Define the parabolic Euler characteristic
$$\chi_{\bm{p}}(E)\,:=\,\chi(E)-\sum_{i=1}^n \sum_{j=1}^{k_{p_i}}r_j(F^{\bullet}(E_{p_i}))\alpha_j(F^{\bullet}(E_{p_i})).$$

Let $\mathcal{E}$ be a family of parabolic bundles on $C$, parametrized by a scheme $T$, of rank $r$, weight 
${\bm{\alpha}}$ and flag type ${\bm r}$. For each point $p_i$, we have a string of rational numbers
$$\bm\alpha_{p_i}\, =\, (0\,\leq\,\alpha_1(p_i)\,<\,\cdots\, <\,\alpha_j(p_i)\, <\,\cdots\,<
\,\alpha_{k_{p_i}}(p_i)\,<\,1)$$ which are the parabolic weights.
Observe that the parabolic Euler characteristic $\chi_{\bm {p}}$ remains constant in a connected family.

Let $\ell$ be the least common multiple of all denominators of all the rational numbers 
appearing in $\bm \alpha$.

\begin{definition}
	The {\em parabolic determinant bundle of level $\ell$} on $M_{\bm
		{\alpha}}^{par}$ is the element of the rational
	Picard group $\operatorname{Pic}(T)_{\mathbb{Q}}$
	given by
	\begin{equation}
		\operatorname{ParDet}(\mathcal{E},\bm{\alpha})\, :=\,
		\operatorname{Det}(\mathcal{E})^{ \ell}\bigotimes\left( \bigotimes_{i=1}^n\left(\bigotimes_{j=1}^{k_{p_i}}
		\det \operatorname{Gr}^j \mathscr{F}_{\bullet, p_i}(\mathcal{E}_{|T\times p_i})\right)^{\ell.\alpha_j(p_i)}\right)
		\otimes \left(\det \mathcal{E}_{|T\times p_0}\right)^{\frac{\ell.\chi_{\bm{p}}(\mathcal{E})}{r}},
	\end{equation}
	where $p_0$ is a fixed point of $C$, and $\Gr^j$ denotes the $j$-th graded
	piece  of the filtration $\mathscr{F}_{\bullet, p_i}$  on
	$\mathcal{E}_{|T\times p_i}$ (cf.\  \cite[Prop.\ 4.5]{BiswasRaghavendra}). 
\end{definition}

Let $M_{\SL_r, \bm{\alpha}}^{par,ss}$ be the moduli space of semistable parabolic $\SL_r$ 
bundles or equivalently parabolic bundles with trivialized determinant. Then
$\operatorname{ParDet}(\mathcal{E},\bm {\alpha})$ descends to a line bundle on 
$M_{\SL_r, \bm {\alpha}}^{par,ss}$, which will be denoted by 
$\operatorname{Det}_{par}(\bm{\alpha})$.

\subsection{Parabolic $G$ bundles}\label{sec:parabolicmoduli}

We shall follow the notation  in \cite[App. A]{BMW1} and refer the reader there for more 
details. Consider the fundamental alcove $\Phi_0$, and let $\bm \tau\,=\,(\tau_1,\,\cdots,\, 
\tau_n)$ be a choice of $n$-tuple of weights in $\Phi_0$ which will be referred to as {\em 
	parabolic weights}.

\begin{definition}Let $G$ be a  connected complex reductive group. A parabolic structure on a
	principal $G$--bundle $E\,\to\, C$ with parabolic structures at the points
	$\bm p\,=\,(p_1,\,\cdots,\, p_n)$ is a choice of parabolic weights $\bm \tau$ along
	with a section $\sigma_i$ of the homogeneous space $E_{p_i}/P(\tau_i)$, for each
	$1\,\leq\, i\,\leq\, n$, where $P(\tau_i)$ is the standard parabolic associated to
	$\tau_i$. Throughout this paper we will assume that $\theta_{\mathfrak{g}}(\tau_i)\,<\,1$
	for all $1\,\leq\, i \,\leq\, n$. 
\end{definition}

We observe that when $G\,=\,\GL_r$, the associated bundle constructed via the standard 
representation of $\GL_r$ recovers the notion of parabolic bundles and parabolic weights as 
in the beginning of the present section. The 
notions of stability and semistability for parabolic $G$-bundles appear in the work of 
Bhosle-Ramanathan \cite{DesaleRamanathan}; for $G\,=\,\GL_r$ they coincide
with the notions of stable and semistable parabolic vector bundles.

Let $\bm\tau$ be an $n$-tuple of parabolic weights in the interior of the Weyl alcove of 
$G$. The corresponding moduli space $M^{par,ss}_{G,\bm\tau}$ (respectively, 
$M^{par,s}_{G,\bm\tau}$) of semistable (respectively, stable) parabolic $G$-bundles was constructed 
in \cite{BalajiBiswasNagaraj}. These moduli spaces are 
normal irreducible quasi-projective varieties. The smooth locus of $M^{par,ss}_{G,\bm\tau}$ 
is denoted by $M^{par,rs}_{G,\bm\tau}$ and it parametrizes regularly stable parabolic 
bundles \cite{BiswasHoffmann:12} or equivalently stable parabolic bundles with minimal automorphisms. 

Let $\iota\,:\, G\,\to\, G'$ be an embedding of connected semisimple groups. 
This homomorphism $\iota$ produces a map $M^{par,ss}_{G,\bm\tau}\,\to\, 
M^{par,ss}_{G',\bm\tau'}$ which is a finite morphism. The weights $\bm \tau'$ and $\bm 
\tau$ are related by $\iota$. This plays a key role in construction of the moduli spaces.
In fact, choosing an appropriate representation of the group $G$, one can reduce the
question of construction to the corresponding question on parabolic vector bundles.

\begin{remark}
	Let $\mathcal{C}\,\to\, S$ be a family of smooth
	curves with $n$ disjoint sections. We will denote the corresponding semistable and
	regularly stable moduli spaces also by $M^{par,ss}_{G,\bm \tau}$ and
	$M^{par,rs}_{G,\bm\tau}$ respectively. When there is a scope of
	confusion, for any $n$-pointed smooth curve $(C,\,\bm p)$, we will use the notation
	$M^{par,ss}_{G,\bm\tau}(C,\,\bm p)$ and $M^{par,rs}_{G,\bm\tau}(C,\,\bm p)$ respectively.
\end{remark}

\subsection{Parabolic bundles as equivariant bundles}\label{sec:orbcorr}

We now discuss parabolic bundles from the point of view of equivariant bundles. We refer the 
reader to  \cite{BalajiSeshadri}, 
\cite{Biswasduke}, \cite{SeshadriI}, and \cite[App. B]{BMW1} for more details. This 
was used in \cite{BMW1} to construct a Hitchin type connection for parabolic 
bundles and it will be crucial here as well. 

\begin{definition}
	Let $p\,:\,\widehat{C}\,\to\, C$ be a Galois cover of curves with Galois group 
	$\Gamma$. A $(\Gamma,\,G)$-bundle is a principal $G$-bundle $\widehat{E}$ on 
	$\widehat{C}$ together with a lift of the action of $\Gamma$ on $\widehat{E}$ as bundle 
	automorphism that commutes with the action of $G$ on $\widehat{E}$.
\end{definition}

Assume that the map $p\,:\, \widehat{C}\,\to\, C$ is ramified over $p_i\, 
\in\,C$, $1\,\leq\, i\,\leq\, n$. Let $\Gamma_{q_i}\,\subset\, \Gamma\,=\, {\rm Gal}(p)$ be 
the isotropy subgroup for some $q_i$ over $p_i$. A $(\Gamma,\, G)$--bundle on a formal disc 
around $q_i$ is uniquely determined by the conjugacy class of a homomorphism 
$\rho_i\,:\,\Gamma_{q_i}\,\to\, G$ given by the action of $\Gamma_{q_i}$ on the 
fiber of the principal $G$--bundle over the point $q_i$ (see 
\cite{BalajiSeshadri,TelWood}). Fix a generator $\gamma_i$ of the cyclic group 
$\Gamma_{q_i}$. Now consider a string of parabolic weights 
${\bm\tau}\,=\,(\tau_1,\,\cdots,\, \tau_n)$ such that $\rho_i(\gamma_i)$ is conjugate to 
$\tau_i$ for each $1\,\leq\, i\,\leq\, n$. We will refer to this $\bm \tau$ as the local 
type of a $(\Gamma,\, G)$--bundle.

The notions of stability and semistability for $(\Gamma,\, G)$--bundles are similar 
to those for the usual principal $G$-bundles; more precisely, the inequality is checked only for the
$\Gamma$ equivariant reductions of the structure group to a parabolic subgroup of $G$
(\cite{BalajiBiswasNagaraj, Ramanathan:75}). Let $M^{\bm \tau, ss}_{G}$ (respectively,
$M^{\bm \tau,s}_{G}$) denote the moduli spaces of semistable (respectively, stable)
$(\Gamma,\,G)$ bundles of local type $\bm \tau$.

Recall the isomorphism $\nu_{\gfrak} : \gfrak^\vee\isorightarrow \gfrak$
from the Killing form.
Given a string of parabolic weights $\bm \tau\,=\,(\tau_1,\,\cdots, \,\tau_n)$, 
choose a minimal integer $\ell$ such that $\exp\left( 2\pi\sqrt{-1}\left(\ell\cdot 
\nu_{\mathfrak{g}}(\tau_i)\right)\right)\,=\,1$. Then by \cite{Namba, Selberg}, we can find a 
ramified Galois cover $p\,:\,\widehat{C}\,\to\, C$ with ramification
exactly over $n$-points $\{p_i\}_{i=1}^n$ whose
isotropy at any ramification point is a cyclic group of order $\ell$. From now on we will 
restrict ourselves only to such Galois covers. The following theorem is due to 
\cite{BalajiBiswasNagaraj, BalajiSeshadri, Biswasduke}.

\begin{theorem}\label{thm:orbifold}
	Consider the moduli stack $\mathcal{B}un^{\bm \tau}_{\Gamma, G}(\widehat{C})$ of $(\Gamma,
	\,G)$--bundles of fixed local type $\bm \tau$. The invariant direct image functor identifies
	the stack $\mathcal{B}un^{\bm \tau}_{\Gamma, G}(\widehat{C})$ with the moduli stack
	$\mathcal{P}ar_{G}(C,\bm{p},\bm \tau)$ of quasi-parabolic bundles of flag type $\bm \tau$. 
	Moreover, the invariant push-forward functor also induces an isomorphism between the moduli
	spaces $M^{par,ss}_{G,\bm\tau}$ (respectively, $M^{par,s}_{G,\bm\tau}$) and
	$M^{\bm \tau,ss}_{G}$ (respectively, $M^{\bm \tau,s}_{G}$).
\end{theorem}

\subsection{Parabolic determinants as equivariant determinants}\label{sec:paradet}

Consider the moduli space $M_{G}^{\bm \tau, ss}$ of $(\Gamma,\,G)$ bundles associated to a 
Galois cover $p\,:\, \widehat{C} \,\to\, C$
with Galois group $\Gamma$. Let $\widehat{M}_{G}$ be the moduli space of 
semistable principal $G$-bundles on the curve $\widehat{C}$. There is a natural forgetful 
map $M^{par,ss}_{G,\bm \tau}\, \to\, \widehat{M}_{G}$ that  simply forgets
the action of $\Gamma$.

Given a representation $\phi\,:\, G\,\to\,\SL_r$, consider the associated morphism
$\overline{\phi}\,:\,\widehat{M}_{G}\,\to\,\widehat{M}_{\SL_r}$ between the
corresponding moduli spaces. Let $\mathcal{L}$ be the determinant of cohomology line bundle
on $\widehat{M}_{\SL_r}$. Let
$$
\mathcal{L}_{\phi}\, :=\,\overline{\phi}^*\mathcal{L}
$$
be its pullback to $\widehat{M}_{G}$. If $G\,=\,\SL_r$, then $\phi$ can be
taken to be the standard
representation. Now Theorem \ref{thm:orbifold} realizes the moduli space
$M^{par,ss}_{G,\bm\tau}$ of parabolic
bundles as a moduli space $M^{\bm \tau,ss}_{G}$ of $(\Gamma,\,G)$--bundles on
$\widehat{C}$, which maps further into $\widehat{M}_G$ by forgetting the action
of $\Gamma$. Thus using the identification between $M_{G}^{\bm \tau,ss}$ and
$M^{par,ss}_{G,\bm\tau}$, we get a natural line bundle $\mathcal{L}_{\phi}$ on
$M^{par,ss}_{G,\bm\tau}$. 

On the other hand, using the parabolic determinant of cohomology, one can construct natural 
line bundles on $M^{par,ss}_{G,\bm \tau}$ as follows:

Let $\bm \tau\,=\,(\tau_1,\,\cdots,\,\tau_n)$ be a string of parabolic weights such that 
$\theta_{\mathfrak{g}}(\tau_i)\,<\,1$ for all $1\,\leq\, i\,\leq \,n$. Take a faithful
representation $(\phi,\,V)$ of the group $G$ satisfying the following condition:
\begin{itemize}
	\item The local type $\phi (\bm \tau)\,=\,(\phi(\tau_1),\,\cdots,\, \phi(\tau_n))$ is
	rational, and $\theta_{\mathfrak{sl}(V)}(\phi (\tau_i))\,<\,1$.
\end{itemize}
Here, $\theta_{\mathfrak{g}}$ and $\theta_{\mathfrak{sl}(V)}$ are the highest roots of the Lie 
algebras $\mathfrak{g}$ and $\mathfrak{sl}(V)$, respectively. We now recall the definition of 
the parabolic determinant of cohomology for $\operatorname{G}$--bundles.

\begin{definition}\label{def:parabolicdet}
	Let $\mathcal{E}$ be a family of parabolic $G$--bundles on a curve $C$ with $n$-marked points,
	and let $\phi\,:\, \operatorname{G}\,\to\, \SL(V)$ be a faithful representation.
	Then the parabolic $G$-determinant bundle $\operatorname{Det}_{par,\phi}(\bm \tau)$
	with weight $\bm \tau$ is defined to be the line bundle
	$\operatorname{Det}_{par}(\nu_{\mathfrak{sl}(V)}(\phi( \bm \tau)))$.
\end{definition}

The following is recalled from \cite{BiswasRaghavendra}. 

\begin{proposition}\label{prop:identification}
	Let $\ell$ be the order of the stabilizer at each ramification point of the Galois cover $p\,:\,\widehat{C}
	\,\to\, C$ with Galois group $\Gamma$, then under the
	isomorphism in Theorem \ref{thm:orbifold}, the parabolic determinant of cohomology
	is related to $\mathcal{L}_{\phi}$ by the formula 
	$$
	\mathcal{L}_{\phi} \,\,\cong\,\, \left(\operatorname{Det}_{par,\phi}(\bm \tau)
	\right)^{\frac{|\Gamma|}{\ell}},
	$$
	where the $\Gamma$ cover $\widehat{C}$ is determined by the
	parabolic weight data $\nu_{\mathfrak{sl}(V)}(\phi \bm (\tau))$.
\end{proposition}

\section{Uniformization of moduli spaces and conformal blocks}

In this section, following the work of Belkale-Fakhruddin \cite{BF}, Laszlo \cite{Laszlo}, 
and Laszlo-Sorger \cite{LaszloSorger:97}, we discuss the universal isomorphism between the
sections of the parabolic determinant of cohomology bundle and the spaces of conformal blocks. If 
$(C,\,\vec{p})$ is a fixed smooth $n$-pointed curve, this identification is due to 
Beauville-Laszlo \cite{BeauvilleLaszlo:94} ($G\,=\,\SL_r$ and $n\,=\,0$), Faltings 
\cite{Faltings:94}, Kumar-Narasimhan-Ramanathan \cite{KNR:94} (for $n\,=\,0$), Pauly 
\cite{Pauly:96} (for $G\,=\,\SL_r$) and Laszlo-Sorger \cite{LaszloSorger:97}. The result has been 
extended to nodal curves by Belkale-Fakhruddin \cite{BF}. All of the results use a key 
uniformization theorem of Harder \cite{Harder} and Drinfeld-Simpson \cite{DS} in the smooth 
case and its generalization in \cite{BF, BF1} for the nodal case. We mostly follow the 
discussion in \cite[Sec. 6]{BF}.

\subsection{The line bundle on the universal moduli stack}\label{sec:uniformization}

Consider the moduli stack $\mathcal{M}_{g,n}$ parametrizing smooth $n$-pointed curves of 
genus $g$.
Recall from Section \ref{sec:parabolicmoduli} that given a tuple $\bm 
\tau\,=\,(\tau_1,\,\cdots,\, \tau_n)$ in the fundamental Weyl alcove $\Phi$ of a simple Lie algebra 
$\mathfrak{g}$, we have the moduli stack $\mathcal{P}ar_{G}(C,\bm{p},\bm \tau)$ of 
quasi-parabolic $G$ bundles of type $\bm \tau$ on a smooth curve $C$. This 
construction for families of smooth $n$-pointed curves gives
relative moduli stacks $\pi_e\,:\, 
\mathcal{P}ar_G(\bm \tau)\,\to\, \mathcal{M}_{g,n} $ such that for any smooth curve 
$(C,\,\vec{p})$ we have $\pi_e^{-1}(C,\bm{p})\,=\,\mathcal{P}ar_G(C,\bm p, \bm \tau)$.
Throughout this discussion, it is assumed that $\theta_{\mathfrak{g}}(\tau_i)\,<\,1$
for all $1\,\leq\, i\,\leq \,n$.

Following \cite{BF} and \cite{Laszlo}, we construct a line bundle 
$\mathscr{L}_{\vec{\lambda}} \to \mathcal{P}ar_G(\bm \tau)$, such that 
$\pi_*\mathscr{L}_{\vec{\lambda}}\,=\,\mathbb{V}^*_{\vec{\lambda}}(\mathfrak{g},\ell),$ where 
$\vec{\lambda}$ and $\ell$ are related to $\bm \tau$ by the exponential map. The 
construction in \cite{BF} extends to the stable nodal curves.

\subsubsection{The relative affine flag varieties}

Let $\mathcal{C}\,\to\, S$ be a family of smooth $n$-pointed curves, and let 
$S\,=\,\operatorname{Spec}{R}$. Consider the affine curve 
$\mathcal{C}'\,=\,\mathcal{C}-\sqcup_{i=1}^np_i(S)$. Let 
$\mathcal{C}_A\,=\,\mathcal{C}\times_R\operatorname{Spec}(A)$ for an $R$ algebra $A$ and 
similarly define $\mathcal{C}_A'$. Let $\widehat{\mathcal{C}}_A$ denote the completion of 
$\mathcal{C}_A$ along the sections $\bm p$. The sections $\bm p$ induce sections of 
$\widehat{\mathcal{C}}_A$, and $\widehat{\mathcal{C}}_A'$ denotes its complement.

Consider the following:
\begin{enumerate}
	\item $L_{\mathcal{C}',G}(A)\,=\,\operatorname{Mor}_k(\mathcal{C}'_A,\, G)$. 
	\item $\mathcal{LG}(A)\,=\,G(\Gamma(\widehat{\mathcal{C}}_A',\,\mathcal{O}))$.
\end{enumerate}
Each $\tau_i$ determines a parabolic subgroup $P(\tau_i)\,\subset\, G$, and we consider 
the standard parahoric subgroup $\mathcal{P}_{\bm \tau}$ given by the inverse image of 
$\prod_{i=1}^n P(\tau_i)$ under the natural evaluation map $L_G(A)\,\lra\, G^{n}$.
Proposition 6.3 of \cite{BF} shows that the $R$ group $L_{\mathcal{C}',G}$ is relatively
ind-affine and formally smooth with connected integral geometric fibers over 
$\operatorname{Spec}(A)$. Observe that if $n\,=\,1$ and $\lambda\,=\,0$, and
if $t$ is a formal 
coordinate at the marked point $p_1$, then $\mathcal{LG}$ gets identified with the loop 
group $L_G$ and $\mathcal{P}_{\bm \tau}$ is the group of positive loops $L^{+}_G$.

\subsubsection{The central extension}

Faltings (\cite{Faltings:94}, and also \cite[Lemma
8.3 ]{BeauvilleLaszlo:94}, \cite{LaszloSorger:97}) constructed a
projective representation of $\mathcal{LG}$ on
$\mathbb{H}_{\vec{\lambda}}\,=\,R\otimes \left( \bigotimes
\mathcal{H}_{\lambda_i}(\mathfrak{g},\ell)\right)$ whose derivative
coincides with the natural projective action of the Lie algebra of
$\mathcal{LG}$. This gives us a central extension 
\begin{equation}\label{eqn:central}
	1\,\longrightarrow\, \mathbb{G}_m \,\longrightarrow\, \widehat{\mathcal{LG}}\,\longrightarrow\,
	\mathcal{LG}\,\longrightarrow\, 1.
\end{equation}
The extension $\widehat{\mathcal{LG}}$ splits
over $\mathcal{P}_{\bm \tau}$ (see \cite[ Lemma 7.3.5]{SorgerGbundle}), and the central extension $\widehat{\mathcal{LG}}$ is independent of the chosen representations $\vec{\lambda}$. Moreover the extension \eqref{eqn:central} splits over $L_{\mathcal{C}',G}$ (\cite{Sorger99}, \cite[Lemma 6.5]{BF}). 

\subsection{The relative uniformization and parabolic theta functions}

Let $\widehat{\mathcal{P}}_{\bm \tau}\,:=\, \mathcal{P}_{\bm \tau}\times \mathbb{G}_m$. The 
weight vectors $\vec{\lambda}$ give natural characters on $\widehat{\mathcal{P}}_{\bm 
	\tau}$ and the product of characters induces a line bundle 
$$\mathscr{L}_{\vec{\lambda}} \,\lra\, {\mathcal{Q}}_{\bm 
	\tau}\,:=\,\widehat{\mathcal{LG}}/\widehat{\mathcal{P}}_{\bm \tau}.$$ Moreover, from the 
uniformization theorems \cite{BF,BeauvilleLaszlo:94, DS, Harder}, it follows that the 
quotient of ${\mathcal{Q}}_{\bm \tau}$ by $L_{\mathcal{C}',G}$ is isomorphic to the pullback $\mathcal{P}ar_G(\bm \tau)_S$ of the stack $\mathcal{P}ar_G(\bm \tau)$ to $S$. Now 
since the extension in \eqref{eqn:central} splits over $L_{\mathcal{C}',G}$, the line bundle 
$\mathscr{L}_{\vec{\lambda}}$ descends to a line bundle over the stack $\mathcal{P}ar_G(\bm \tau)$ 
which we will also denote by $\mathscr{L}_{\vec{\lambda}}$. Observe that the line bundle 
$\mathscr{L}_{\vec{\lambda}}$ is trivialized along the trivial 
section of $\mathcal{P}ar_G(\bm \tau)$ over $S$, and such data determine the line bundle 
up to canonical isomorphism. We will refer to the line bundle $\mathscr{L}_{\vec{\lambda}}$ 
as the {\em Borel-Weil-Bott} line bundle.

\subsubsection{Parabolic determinant as the Borel-Weil-Bott line bundle}

We now compare the parabolic determinant of cohomology of the universal bundle with the 
line bundle $\mathscr{L}_{\vec{\lambda}}$.

Recall from Definition \ref{def:parabolicdet} the notion of the parabolic determinant 
$\operatorname{Det}_{par,\phi}(\bm \tau)$ of cohomology associated to a family of parabolic 
$G$ bundles on $\mathcal{C}\,\to\,S$ and a suitable representation $\phi\,:
\,G\,\to\, 
\SL(V)$. Now for the fixed $n$-pointed curve $(C,\,\bm p)$, it is known that the line 
bundles $\mathscr{L}^{\otimes m_{\phi}}_{\vec{\lambda}}$ and 
$\operatorname{Det}_{par,\phi}(\bm \tau)$ on $\mathcal{P}ar_G(C,\bm p,\bm \tau)$ are 
isomorphic, where $m_{\phi}$ is the Dynkin index of the embedding $\phi$. Since these 
line bundles are determined up to a normalizing factor, it follows that the corresponding 
projective bundles are identified as
\begin{equation}\label{eqn:impprojisomorphism}
	\mathbb{P}{\pi_{e}}_*\left(\operatorname{Det}_{par, \phi}(\bm \tau)\right)
	\,\cong\, \mathbb{P}{\pi_e}_*\left( \mathscr{L}_{\vec{\lambda}}^{\otimes m_{\phi}}\right),
\end{equation}
where $\pi_{e}: \mathcal{P}ar_{G}(\bm \tau)\,\to\, \mathcal{M}_{g,n}$ is
the natural projection.

\subsubsection{Parabolic theta functions and conformal blocks}

For any choice of formal parameters, the ind-scheme ${\mathcal{Q}}_{\bm \tau}$ can be
identified with 
the product of affine flag varieties $\prod_{i=1}^nL_G/\mathcal{P}_{\tau_i}$ and the line 
bundle $\mathscr{L}_{\vec{\lambda}}$ pulls back to the corresponding line bundle on 
$L_G/\mathcal{P}_{\tau_i}$ given by the character $\lambda_i$. Now by Kumar 
\cite{kumarbook} and Mathieu \cite{Mathieu}, we get that
\begin{equation}\label{eqn:pizzatheorem}
	H^0(\mathcal{Q}_{\bm \tau},\,\,\mathscr{L}_{\vec{\lambda}})\,=\,\mathbb{H}_{\vec{\lambda}}^{*}.
\end{equation}

We end this discussion with the following theorem (see \cite[Theorem 1.7]{BF} and 
\cite[Sec. 5.7]{Laszlo}) which we will refer to as the universal identification of 
the parabolic theta functions and the conformal blocks. In the case when $S$ is a point, the result 
can be found in \cite{BeauvilleLaszlo:94, Faltings:94, KNR:94, LaszloSorger:97}.

\begin{theorem}\label{thm:uniformverlinde}
	The push-forward of $\mathscr{L}_{\vec{\lambda}}$ along the map $\pi_e\,:\,
	\mathcal{P}ar_G(\bm \tau)\,\to\, \mathcal{M}_{g,n}$ can be
	identified canonically with the bundle of coordinate free conformal blocks $\mathbb{V}^{\dagger}_{\vec{\lambda}}(\mathfrak{g},\ell)$. 
	Moreover, 
	$\mathscr{L}_{\vec{\lambda}}$ descends to a
	line bundle on $M_{G,\bm \tau}^{par,rs}$, and $({\pi_e}_{|M^{par,rs}_{G,\bm
			\tau}})_*\mathscr{L}_{\vec{\lambda}}$ is isomorphic to
	$\mathbb{V}_{\vec{\lambda}}^{\dagger}(\mathfrak{g},\ell)$ provided the following conditions hold:
	\begin{itemize}  
		\item The genus of the orbifold curve determined by $\bm \tau$ is at least $2$, if $G$ is 
		not $\SL_2$.
		
		\item The genus of the orbifold curve is at least $3$, if $G=\SL_2$.
	\end{itemize}
\end{theorem}

\begin{remark}
	The last conditions ensure that for any smooth pointed curve $(C,\,\bm p)$, the
	codimension of the moduli space $M_{G,\bm \tau}^{par,rs}(C,\bm p)$ in the moduli stack
	$\mathcal{P}ar_G(C, \bm p,\bm \tau)$ is at least two. We refer the reader to \cite[App. C]{BMW1}.
\end{remark}

\medskip

\medskip
 
\bibliographystyle{amsplain}

\bibliography{./papers}

\providecommand{\MR}[1]{}
\providecommand{\bysame}{\leavevmode\hbox to3em{\hrulefill}\thinspace}
\providecommand{\MR}{\relax\ifhmode\unskip\space\fi MR }
\providecommand{\MRhref}[2]{%
  \href{http://www.ams.org/mathscinet-getitem?mr=#1}{#2}
}
\providecommand{\href}[2]{#2}
\begin{thebibliography}{10}

\bibitem{Andersen:12}
J\o rgen~Ellegaard Andersen, \emph{Hitchin's connection, {T}oeplitz operators,
  and symmetry invariant deformation quantization}, Quantum Topol. \textbf{3}
  (2012), no.~3-4, 293--325. \MR{2928087}

\bibitem{Andersen:personal}
J\o rgen~Ellegaard Andersen and Jens~Kristian Egsgaard, \emph{The equivalence
  of the {H}itchin connection and the {K}nizhnik-{Z}amolodchikov connection},
  personal communication, March 2021.

\bibitem{ADW}
Scott Axelrod, Steve Della~Pietra, and Edward Witten, \emph{Geometric
  quantization of {C}hern-{S}imons gauge theory}, J. Differential Geom.
  \textbf{33} (1991), no.~3, 787--902. \MR{1100212}

\bibitem{BBMP20}
Thomas Baier, Michele Bolognesi, Johan Martens, and Christian Pauly, \emph{The
  {H}itchin connection in arbitrary characteristic}, preprint (2020), 47 pp.,
  arxiv:2002.12288.

\bibitem{BalajiBiswasNagaraj}
Vikraman Balaji, Indranil Biswas, and D.~S. Nagaraj., \emph{Ramified
  {$G$}-bundles as parabolic bundles}, J. Ramanujan Math. Soc. \textbf{18}
  (2003), no.~2, 123--138. \MR{1995862}

\bibitem{BalajiSeshadri}
Vikraman Balaji and C.~S. Seshadri, \emph{Moduli of parahoric
  {$\mathscr{G}$}-torsors on a compact {R}iemann surface}, J. Algebraic Geom.
  \textbf{24} (2015), no.~1, 1--49. \MR{3275653}

\bibitem{BaragliaKamgarpourVarma:19}
David Baraglia, Masoud Kamgarpour, and Rohith Varma, \emph{Complete
  integrability of the parahoric {H}itchin system}, Int. Math. Res. Not. IMRN
  (2019), no.~21, 6499--6528. \MR{4027558}

\bibitem{BeauvilleLaszlo:94}
Arnaud Beauville and Yves Laszlo, \emph{Conformal blocks and generalized theta
  functions}, Comm. Math. Phys. \textbf{164} (1994), no.~2, 385--419.
  \MR{1289330 (95k:14011)}

\bibitem{BD1}
Alexander Beilinson and Vladimir Drinfeld, \emph{Quantization of {H}itchin's
  integrable system and {H}ecke eigensheaves}.

\bibitem{BPZ:84}
Alexander.~A. Belavin, Alexander.~M. Polyakov, and Alexander.~B. Zamolodchikov,
  \emph{Infinite conformal symmetry in two-dimensional quantum field theory},
  Nuclear Phys. B \textbf{241} (1984), no.~2, 333--380. \MR{757857}

\bibitem{Belkale:09}
Prakash Belkale, \emph{Strange duality and the {H}itchin/{WZW} connection}, J.
  Differential Geom. \textbf{82} (2009), no.~2, 445--465. \MR{2520799
  (2010j:14065)}

\bibitem{BF}
Prakash Belkale and Najmuddin Fakhruddin, \emph{Triviality properties of
  principal bundles on singular curves}, Algebr. Geom. \textbf{6} (2019),
  no.~2, 234--259. \MR{3914752}

\bibitem{BF1}
\bysame, \emph{Triviality properties of principal bundles on singular curves.
  {II}}, Canad. Math. Bull. \textbf{63} (2020), no.~2, 423--433. \MR{4092891}

\bibitem{BGM2}
Prakash Belkale, Angela Gibney, and Swarnava Mukhopadhyay, \emph{Vanishing and
  identities of conformal blocks divisors}, Algebr. Geom. \textbf{2} (2015),
  no.~1, 62--90. \MR{3322198}

\bibitem{BGM1}
\bysame, \emph{Nonvanishing of conformal blocks divisors on {$\overline
  M_{0,n}$}}, Transform. Groups \textbf{21} (2016), no.~2, 329--353.
  \MR{3492039}

\bibitem{DesaleRamanathan}
Usha Bhosle and A.~Ramanathan, \emph{Moduli of parabolic {$G$}-bundles on
  curves}, Math. Z. \textbf{202} (1989), no.~2, 161--180. \MR{1013082}

\bibitem{Biswasduke}
Indranil Biswas, \emph{Parabolic bundles as orbifold bundles}, Duke Math. J.
  \textbf{88} (1997), no.~2, 305--325. \MR{1455522}

\bibitem{BiswasHoffmann:12}
Indranil Biswas and Norbert Hoffmann, \emph{Poincar\'{e} families of
  {$G$}-bundles on a curve}, Math. Ann. \textbf{352} (2012), no.~1, 133--154.
  \MR{2885579}

\bibitem{BMW1}
Indranil Biswas, Swarnava Mukhopadhyay, and Richard Wentworth, \emph{A
  {H}itchin connection on non abelian theta functions for parabolic
  {$G$}-bundles}, to appear in Crelle (2023), 58 pp., arxiv:2103.03792.

\bibitem{BiswasRaghavendra}
Indranil Biswas and Nyshadham Raghavendra, \emph{Determinants of parabolic
  bundles on {R}iemann surfaces}, Proc. Indian Acad. Sci. Math. Sci.
  \textbf{103} (1993), no.~1, 41--71. \MR{1234199}

\bibitem{BjerreThesis}
Mette Bjerre, \emph{The {H}itchin connection for the quantization of the moduli
  space of parabolic bundles on surfaces with marked points}, PhD Thesis
  (2018), 101 pp., Aarhus University.

\bibitem{BY}
Hans~U. Boden and K\^{o}ji Yokogawa, \emph{Moduli spaces of parabolic {H}iggs
  bundles and parabolic {$K(D)$} pairs over smooth curves. {I}}, Internat. J.
  Math. \textbf{7} (1996), no.~5, 573--598. \MR{1411302}

\bibitem{Drinfeld}
Vladimir~G. Drinfeld, \emph{Quantum groups}, Proceedings of the {I}nternational
  {C}ongress of {M}athematicians, {V}ol. 1, 2 ({B}erkeley, {C}alif., 1986),
  Amer. Math. Soc., Providence, RI, 1987, pp.~798--820. \MR{934283}

\bibitem{DS}
Vladimir~G. Drinfeld and Carlos Simpson, \emph{{$B$}-structures on
  {$G$}-bundles and local triviality}, Math. Res. Lett. \textbf{2} (1995),
  no.~6, 823--829. \MR{1362973}

\bibitem{Fakhruddin:12}
Najmuddin Fakhruddin, \emph{Chern classes of conformal blocks}, Compact moduli
  spaces and vector bundles, Contemp. Math., vol. 564, Amer. Math. Soc.,
  Providence, RI, 2012, pp.~145--176. \MR{2894632}

\bibitem{Faltings:93}
Gerd Faltings, \emph{Stable {$G$}-bundles and projective connections}, J.
  Algebraic Geom. \textbf{2} (1993), no.~3, 507--568. \MR{1211997 (94i:14015)}

\bibitem{Faltings:94}
\bysame, \emph{A proof for the {V}erlinde formula}, J. Algebraic Geom.
  \textbf{3} (1994), no.~2, 347--374. \MR{1257326 (95j:14013)}

\bibitem{FSV2}
Boris Feigin, Vadim Schechtman, and Alexander Varchenko, \emph{On algebraic
  equations satisfied by hypergeometric correlators in {WZW} models. {II}},
  Comm. Math. Phys. \textbf{170} (1995), no.~1, 219--247. \MR{1331699}

\bibitem{FSV1}
Boris Feigin, Vadim~V. Schechtman, and Alexander Varchenko, \emph{On algebraic
  equations satisfied by hypergeometric correlators in {WZW} models. {I}},
  Comm. Math. Phys. \textbf{163} (1994), no.~1, 173--184. \MR{1277938}

\bibitem{Ginzburg}
Victor Ginzburg, \emph{Resolution of diagonals and moduli spaces}, The moduli
  space of curves ({T}exel {I}sland, 1994), Progr. Math., vol. 129,
  Birkh\"{a}user Boston, Boston, MA, 1995, pp.~231--266. \MR{1363059}

\bibitem{Harder}
G\"{u}nter Harder, \emph{Halbeinfache {G}ruppenschemata \"{u}ber
  {D}edekindringen}, Invent. Math. \textbf{4} (1967), 165--191. \MR{225785}

\bibitem{Hitchin:90}
Nigel Hitchin, \emph{Flat connections and geometric quantization}, Comm. Math.
  Phys. \textbf{131} (1990), no.~2, 347--380. \MR{1065677 (91g:32022)}

\bibitem{KacWakimoto:88}
Victor~G. Kac and Minoru Wakimoto, \emph{Modular and conformal invariance
  constraints in representation theory of affine algebras}, Adv. in Math.
  \textbf{70} (1988), no.~2, 156--236. \MR{954660 (89h:17036)}

\bibitem{KZ}
Vadim~G. Knizhnik and Alexander~B. Zamolodchikov, \emph{Current algebra and
  {W}ess-{Z}umino model in two dimensions}, Nuclear Phys. B \textbf{247}
  (1984), no.~1, 83--103. \MR{853258}

\bibitem{Kohno}
Toshitake Kohno, \emph{Monodromy representations of braid groups and
  {Y}ang-{B}axter equations}, Ann. Inst. Fourier (Grenoble) \textbf{37} (1987),
  no.~4, 139--160. \MR{927394}

\bibitem{kumarbook}
Shrawan Kumar, \emph{Demazure character formula in arbitrary {K}ac-{M}oody
  setting}, Invent. Math. \textbf{89} (1987), no.~2, 395--423. \MR{894387}

\bibitem{KNR:94}
Shrawan Kumar, M.~S. Narasimhan, and A.~Ramanathan, \emph{Infinite
  {G}rassmannians and moduli spaces of {$G$}-bundles}, Math. Ann. \textbf{300}
  (1994), no.~1, 41--75. \MR{1289830 (96e:14011)}

\bibitem{KumarNarasimhan:97}
Shrawan Kumar and M.S. Narasimhan, \emph{Picard group of the moduli spaces of
  {$G$}-bundles}, Math. Ann. \textbf{308} (1997), no.~1, 155--173. \MR{1446205
  (98d:14014)}

\bibitem{Laszlo}
Yves Laszlo, \emph{Hitchin's and {WZW} connections are the same}, J.
  Differential Geom. \textbf{49} (1998), no.~3, 547--576. \MR{1669720
  (2000e:14012)}

\bibitem{LaszloSorger:97}
Yves Laszlo and Christoph Sorger, \emph{The line bundles on the moduli of
  parabolic {$G$}-bundles over curves and their sections}, Ann. Sci. \'Ecole
  Norm. Sup. (4) \textbf{30} (1997), no.~4, 499--525. \MR{1456243 (98f:14007)}

\bibitem{LogaresMartens:10}
Marina Logares and Johan Martens, \emph{Moduli of parabolic {H}iggs bundles and
  {A}tiyah algebroids}, J. Reine Angew. Math. \textbf{649} (2010), 89--116.
  \MR{2746468}

\bibitem{Mathieu}
Olivier Mathieu, \emph{Formules de caract\`eres pour les alg\`ebres de
  {K}ac-{M}oody g\'{e}n\'{e}rales}, Ast\'{e}risque (1988), no.~159-160, 267.
  \MR{980506}

\bibitem{MehtaSeshadri}
Vikram Mehta and C.S. Seshadri, \emph{Moduli of vector bundles on curves with
  parabolic structures}, Math. Ann. \textbf{248} (1980), no.~3, 205--239.
  \MR{575939}

\bibitem{MooreSeiberg:90}
Gregory Moore and Nathan Seiberg, \emph{Lectures on {RCFT}}, Superstrings '89
  ({T}rieste, 1989), World Sci. Publ., River Edge, NJ, 1990, pp.~1--129.
  \MR{1159969}

\bibitem{MW}
Swarnava Mukhopadhyay and Richard Wentworth, \emph{Generalized theta functions,
  strange duality, and odd orthogonal bundles on curves}, Comm. Math. Phys.
  \textbf{370} (2019), no.~1, 325--376. \MR{3982698}

\bibitem{Namba}
Makoto Namba, \emph{Branched coverings and algebraic functions}, Pitman
  Research Notes in Mathematics Series, vol. 161, Longman Scientific \&
  Technical, Harlow; John Wiley \& Sons, Inc., New York, 1987. \MR{933557}

\bibitem{NarSesh}
M.~S. Narasimhan and C.~S. Seshadri, \emph{Stable and unitary vector bundles on
  a compact {R}iemann surface}, Ann. of Math. (2) \textbf{82} (1965), 540--567.
  \MR{184252}

\bibitem{Novikov:82}
Sergei.~P. Novikov, \emph{The {H}amiltonian formalism and a multivalued
  analogue of {M}orse theory}, Uspekhi Mat. Nauk \textbf{37} (1982),
  no.~5(227), 3--49, 248. \MR{676612}

\bibitem{Pauly:96}
Christian Pauly, \emph{Espaces de modules de fibr\'es paraboliques et blocs
  conformes}, Duke Math. J. \textbf{84} (1996), no.~1, 217--235. \MR{1394754
  (97h:14022)}

\bibitem{RamaHitchin}
T.~R. Ramadas, \emph{Faltings' construction of the {K}-{Z} connection}, Comm.
  Math. Phys. \textbf{196} (1998), no.~1, 133--143. \MR{1643521}

\bibitem{Ramanathan:75}
A.~Ramanathan, \emph{Stable principal bundles on a compact {R}iemann surface},
  Math. Ann. \textbf{213} (1975), 129--152. \MR{369747}

\bibitem{Ran}
Ziv Ran, \emph{Jacobi cohomology, local geometry of moduli spaces, and
  {H}itchin connections}, Proc. London Math. Soc. (3) \textbf{92} (2006),
  no.~3, 545--580. \MR{2223536}

\bibitem{Reshetikhin}
Nicolai Reshetikhin, \emph{The {K}nizhnik-{Z}amolodchikov system as a
  deformation of the isomonodromy problem}, Lett. Math. Phys. \textbf{26}
  (1992), no.~3, 167--177. \MR{1199740}

\bibitem{SVCoh}
Vadim~V. Schechtman and Alexander~N. Varchenko, \emph{Arrangements of
  hyperplanes and {L}ie algebra homology}, Invent. Math. \textbf{106} (1991),
  no.~1, 139--194. \MR{1123378}

\bibitem{SS}
Peter Scheinost and Martin Schottenloher, \emph{Metaplectic quantization of the
  moduli spaces of flat and parabolic bundles}, J. Reine Angew. Math.
  \textbf{466} (1995), 145--219. \MR{1353317}

\bibitem{Selberg}
Atle Selberg, \emph{On discontinuous groups in higher-dimensional symmetric
  spaces}, Contributions to function theory (internat. {C}olloq. {F}unction
  {T}heory, {B}ombay, 1960), Tata Institute of Fundamental Research, Bombay,
  1960, pp.~147--164. \MR{0130324}

\bibitem{SeshadriI}
C.~S. Seshadri, \emph{Moduli of {$\pi$}-vector bundles over an algebraic
  curve}, Questions on {A}lgebraic {V}arieties ({C}.{I}.{M}.{E}., {III}
  {C}iclo, {V}arenna, 1969), Edizioni Cremonese, Rome, 1970, pp.~139--260.
  \MR{0280496}

\bibitem{Sorger}
Christoph Sorger, \emph{La formule de {V}erlinde}, no. 237, 1996, S\'{e}minaire
  Bourbaki, Vol. 1994/95, pp.~Exp. No. 794, 3, 87--114. \MR{1423621}

\bibitem{Sorger99}
\bysame, \emph{On moduli of {$G$}-bundles of a curve for exceptional {$G$}},
  Ann. Sci. \'{E}cole Norm. Sup. (4) \textbf{32} (1999), no.~1, 127--133.
  \MR{1670528}

\bibitem{SorgerGbundle}
\bysame, \emph{Lectures on moduli of principal {$G$}-bundles over algebraic
  curves}, School on {A}lgebraic {G}eometry ({T}rieste, 1999), ICTP Lect.
  Notes, vol.~1, Abdus Salam Int. Cent. Theoret. Phys., Trieste, 2000,
  pp.~1--57. \MR{1795860}

\bibitem{ST}
Xiaotao Sun and I-Hsun Tsai, \emph{Hitchin's connection and differential
  operators with values in the determinant bundle}, J. Differential Geom.
  \textbf{66} (2004), no.~2, 303--343. \MR{2106127}

\bibitem{Teleman}
Constantin Teleman, \emph{Lie algebra cohomology and the fusion rules}, Comm.
  Math. Phys. \textbf{173} (1995), no.~2, 265--311. \MR{1355626}

\bibitem{TelWood}
Constantin Teleman and Christopher Woodward, \emph{Parabolic bundles, products
  of conjugacy classes and {G}romov-{W}itten invariants}, Ann. Inst. Fourier
  (Grenoble) \textbf{53} (2003), no.~3, 713--748. \MR{2008438}

\bibitem{Tsuchimoto}
Yoshifumi Tsuchimoto, \emph{On the coordinate-free description of the conformal
  blocks}, J. Math. Kyoto Univ. \textbf{33} (1993), no.~1, 29--49. \MR{1203889}

\bibitem{TK88}
Akihiro Tsuchiya and Yukihiro Kanie, \emph{Vertex operators in conformal field
  theory on {${\bf P}^1$} and monodromy representations of braid group},
  Conformal field theory and solvable lattice models ({K}yoto, 1986), Adv.
  Stud. Pure Math., vol.~16, Academic Press, Boston, MA, 1988, pp.~297--372.
  \MR{972998}

\bibitem{TUY:89}
Akihiro Tsuchiya, Kenji Ueno, and Yasuhiko Yamada, \emph{Conformal field theory
  on universal family of stable curves with gauge symmetries}, Integrable
  systems in quantum field theory and statistical mechanics, Adv. Stud. Pure
  Math., vol.~19, Academic Press, Boston, MA, 1989, pp.~459--566. \MR{1048605
  (92a:81191)}

\bibitem{VanGeemenDeJong:98}
Bert van Geemen and Aise~Johan de~Jong, \emph{On {H}itchin's connection}, J.
  Amer. Math. Soc. \textbf{11} (1998), no.~1, 189--228. \MR{1469656}

\bibitem{welters}
Gerald~E. Welters, \emph{Polarized abelian varieties and the heat equations},
  Compositio Math. \textbf{49} (1983), no.~2, 173--194. \MR{704390}

\bibitem{Witten:84}
Edward Witten, \emph{Nonabelian bosonization in two dimensions}, Comm. Math.
  Phys. \textbf{92} (1984), no.~4, 455--472. \MR{736403}

\bibitem{Witten89}
\bysame, \emph{Quantum field theory and the {J}ones polynomial}, Braid group,
  knot theory and statistical mechanics, Adv. Ser. Math. Phys., vol.~9, World
  Sci. Publ., Teaneck, NJ, 1989, pp.~239--329. \MR{1062429}

\end{thebibliography}

\end{document}